\documentclass[12pt]{amsart}

\usepackage[utf8]{inputenc}
\usepackage[T1]{fontenc}
\usepackage{lmodern}
\usepackage{amsmath,amssymb,amsthm}
\usepackage{mathtools}
\mathtoolsset{showonlyrefs=true}
\usepackage[colorlinks=true,linkcolor=blue,citecolor=blue,urlcolor=blue]{hyperref}
\usepackage[margin=1in]{geometry}
\newtheorem{theorem}{Theorem}[section]
\newtheorem{proposition}[theorem]{Proposition}
\newtheorem{lemma}[theorem]{Lemma}
\newtheorem{corollary}[theorem]{Corollary}
\theoremstyle{definition}
\newtheorem{definition}[theorem]{Definition}
\newtheorem{remark}[theorem]{Remark}
\newtheorem{note}[theorem]{Note}
\newtheorem{example}[theorem]{Example}

\DeclareMathOperator{\GL}{GL}

\DeclareMathOperator{\Frac}{Frac}
\DeclareMathOperator{\im}{Im}

\newcommand{\F}{\mathbb{F}}
\newcommand{\StDeltai}{\mathrm{St}^{\Delta_i}}
\newcommand{\Qns}{Q_{n,s}}
\newcommand{\Qnzero}{Q_{n,0}}
\newcommand{\Ln}{L_n}
\newcommand{\Lns}{L_{n,s}}

\newcommand{\SumCite}{\cite{Sum}}

\def\DD{D\kern-.7em\raise0.4ex\hbox{\char '55}\kern.33em}

\title[Normalized derivations for Steenrod--Milnor operations]{Normalized Derivations for Steenrod--Milnor Operations on the Dickson Algebra and Applications}

\author{\DD\d{\u a}ng V\~o Ph\'uc}
\address{Department of Mathematics, FPT University, Quy Nhon AI Campus, An Phu Thinh New Urban Area, Vietnam}
\email{dangphuc150488@gmail.com, phucdv14@fpt.edu.vn}
\thanks{ORCID: \url{https://orcid.org/0000-0002-6885-3996}}


\begin{document}

\begin{abstract}
Understanding the action of the Steenrod algebra on the Dickson algebra $D_n$ is a central problem in modular invariant theory and algebraic topology. While explicit first-order formulas for the primitive Steenrod--Milnor operations $\StDeltai$ have been established, computing their higher iterates has historically been obstructed by severe combinatorial complexity. In this paper, we resolve this obstruction for all $i\ge 1$ by introducing a normalized derivation $\delta_i=(-1)^nQ_{n,0}^{-1}\StDeltai$. Starting from an explicit formula in a recent paper by Nguyen Sum, we show that $\delta_i$ preserves $D_n$ and elegantly transforms the non-linear action into a tractable affine differential system.

By resolving the non-commutativity of the operators via unsigned Stirling numbers of the first kind, we derive closed formulas for all higher iterates $\delta_i^{\,m}(Q_{n,s})$ and $(\StDeltai)^m(Q_{n,s})$, which systematically yield the ambient nilpotence identity $(\StDeltai)^p=0$ on $D_n$. Furthermore, we discover that the normalized derivation satisfies a restricted $p$-power relation $\delta_i^{\,p}=B^{p-1}\delta_i$, where $B=R_{n,i}^{\,p}$. When $B\ne0$, localizing at $B$ allows us to construct an explicit $\F_p$-weight decomposition of the localized Dickson algebra using Lagrange projectors. This multiplicative decomposition provides exact descriptions of the kernels and images of both $\delta_i$ and $\StDeltai$.

As applications, we completely determine the kernel and image of $\delta_i$ in the classical range $1\le i<n$, and describe them via an auxiliary polynomial grading for $i=n$. The general formalism unifies and upgrades the classical first-order formulas of Smith--Switzer and Wilkerson, as well as Sum's formulas for $i=n+1,n+2$, into exact closed expressions for all higher iterates. Finally, we formulate an ordinary commutative-algebra Koszul complex associated with the normalized-ratio coefficients, providing a structural clarification against the square-zero Margolis complexes recently studied in the literature.
\end{abstract}

\maketitle
\enlargethispage{2pt}

\section{Introduction}

Let $p$ be an odd prime and $n\ge 1$, and let $P_n=\F_p[x_1,\dots,x_n]$ be the polynomial algebra graded by $|x_j|=2$. The general linear group $\GL_n(\F_p)$ acts naturally on $P_n$ through its linear action on the vector space spanned by $x_1,\dots,x_n$. In a foundational work, Dickson \cite{Dickson} showed that the invariant algebra $P_n^{\GL_n(\F_p)}$ is a polynomial algebra generated by the so-called Dickson invariants $Q_{n,0},Q_{n,1},\dots,Q_{n,n-1}$. Accordingly, we write
\[
D_n:=P_n^{\GL_n(\F_p)}=\F_p[\Qnzero,Q_{n,1},\dots,Q_{n,n-1}]
\]
and refer to it as the Dickson algebra. The interaction between modular invariants and Steenrod operations has been a subject of extensive study, with significant contributions from Adams--Wilkerson \cite{AdamsWilkerson,AdamsWilkersonCorrection}, Smith--Switzer \cite{SmithSwitzer}, Wilkerson \cite{WilkersonPrimer}, M\`ui \cite{Mui1986}, and Sum \cite{Sum1992,Sum1994,Sum2007,Sum}. 

For an odd prime $p$, Milnor \cite{Milnor} identified the dual mod-$p$ Steenrod algebra as
\[
\mathcal A_p^*\cong E(\tau_0,\tau_1,\dots)\otimes P(\xi_1,\xi_2,\dots).
\]
Following Sum's notation \cite{Sum2007,Sum}, for every $i\ge 1$ we write
\[
\StDeltai=\mathrm{St}^{\emptyset,\Delta_i},
\qquad
\Delta_i=(0,\dots,0,1,0,\dots),
\]
where the $1$ occurs in the $i$-th position. Thus $\StDeltai$ is the Steenrod--Milnor operation dual to the monomial $\xi_i$. It is crucial to distinguish this from the usual Milnor primitive $Q_i$, which lies on the exterior $\tau$-side and is dual to $\tau_i$. The operation $\StDeltai$ is primitive and therefore acts as a derivation on $P_n$. Since $P_n$ is concentrated in even degrees and $\StDeltai$ has even degree, no Koszul sign enters the Leibniz rule. The point of departure for the present work is Sum's explicit formula \cite[Theorem~2.2]{Sum}, which states that for $0\le s<n$ and $i\ge1$,
\[
\StDeltai(\Qns)=(-1)^n \Qnzero\Big(P_{n,i,s}^{\,p}+R_{n,i}^{\,p}\,\Qns\Big).
\]

The systematic appearance of the common factor $\Qnzero$ motivates the main conceptual leap of this paper. By dividing out this factor, we introduce the normalized operator
\[
\delta_i:=(-1)^n\Qnzero^{-1}\StDeltai
\]
on the localization $D_n[\Qnzero^{-1}]$. Because $\StDeltai$ is a derivation, $\delta_i$ is a genuine $\F_p$-linear derivation, and Sum's formula implies that it actually preserves $D_n$. By writing $A_s=P_{n,i,s}^{\,p}$ and $B=R_{n,i}^{\,p}$, the action becomes a clean affine differential system: $\delta_i(Q_{n,s})=A_s+BQ_{n,s}$.

This normalized-derivation viewpoint unlocks several deeper algebraic structures. First, by expanding the non-commuting product of $Q_{n,0}$ and $\delta_i$ via unsigned Stirling numbers of the first kind, we derive a closed formula for all higher iterates of $\StDeltai$, culminating in the ambient nilpotence identity $(\StDeltai)^p=0$. Second, and more strikingly, we discover that the normalized derivation satisfies a restricted $p$-power relation: $\delta_i^p=B^{p-1}\delta_i$. When $B\ne0$, localizing further by inverting $B$ turns the rescaled derivation $\partial_i:=B^{-1}\delta_i$ into an idempotent operator ($\partial_i^p=\partial_i$). This allows us to construct an explicit $\F_p$-weight decomposition of the localized Dickson algebra using Lagrange projectors, completely determining the kernel and image of the action. 

Finally, this paper addresses the structural role of Dickson algebras in homological contexts. Recent works have utilized invariant algebras to compute Margolis homology (e.g., Hung \cite{Hung2020,Hung2021} and Tuan \cite{Tuan}). Those calculations strictly rely on the square-zero Milnor primitives $Q_j$. By contrast, since $\StDeltai$ lies on the polynomial $\xi$-side and is not generally square-zero, we show that our normalized ratios naturally give rise to an ordinary commutative-algebra Koszul complex, framing a careful structural comparison with Margolis complexes.

In summary, the main contributions of this paper are the following:
\begin{itemize}
    \item We introduce the normalized derivation $\delta_i$ and establish the foundational affine differential system on the Dickson generators.
    \item We overcome the combinatorial complexity of the action by using Stirling numbers to derive a closed formula for $(\StDeltai)^m(Q_{n,s})$ for all $m \ge 1$, which systematically yields the operator identity $(\StDeltai)^p=0$ on $D_n$.
    \item We prove the restricted $p$-power relation $\delta_i^p=B^{p-1}\delta_i$ and construct a multiplicative $\F_p$-weight decomposition of the localized Dickson algebra via explicit Lagrange projectors.
    \item We provide exact kernel and image descriptions for both $\delta_i$ and $\StDeltai$, utilizing the weight decomposition and invariant ratios.
    \item We show that the general formalism unifies, recovers, and upgrades classical first-order formulas (for $1\le i\le n$) and Sum's low extra formulas (for $i=n+1,n+2$) into closed expressions for all higher iterates.
    \item We construct a formal Koszul complex associated with the normalized-ratio coefficients and clarify its algebraic distinction from Margolis homology.
\end{itemize}

The paper is organized as follows. Section~\ref{sec:main} recalls the determinant notation, introduces the normalized derivation framework, and states all principal theorems, including the iterate formulas, the restricted $p$-power relation, and the $\F_p$-weight decomposition. Section~\ref{sec:proofs} is dedicated to the rigorous proofs; it details the algebraic manipulations involving Stirling numbers for the iterate formulas and constructs the Lagrange projectors for the weight spaces. Section~\ref{sec:explicit} applies this general machinery to specific ranges, upgrading known first-order identities for the classical range ($1 \le i \le n$) and low extra range ($i=n+1, n+2$) to their higher-iterate counterparts. Finally, Section~\ref{sec:koszul} develops the associated Koszul constructions and contrasts their commutative-algebra mechanism with existing Margolis-homology literature.

\section{Preliminaries and main results}\label{sec:main}

This section establishes the algebraic framework used throughout the paper. We first recall the determinant definitions of the Dickson invariants and Sum's formula for the primitive operations. We then introduce the normalized derivation and state its chain-rule, iterate, restricted-power, ratio, weight-decomposition, kernel, and image consequences. The operator identities and the remaining proofs are given in Section~\ref{sec:proofs}.

\subsection{Dickson invariants and the Sum formula}

We begin by fixing notation, following Dickson \cite{Dickson} and Sum \cite[\S 2]{Sum}. For integers $e_1,\dots,e_n\ge 0$, let
\[
[e_1,\dots,e_n]
:=
\det\big(x_r^{p^{e_s}}\big)_{1\le r,s\le n}
\in P_n.
\]
Following Dickson, put
\[
\Ln=[0,1,\dots,n-1],
\qquad
\Lns=[0,1,\dots,\widehat{s},\dots,n]
\quad(0\le s<n),
\]
where the hat means omission. The Dickson invariants are then defined by
\[
\Qns=\frac{\Lns}{\Ln}
\quad(0\le s<n),
\qquad
\Qnzero=\Ln^{\,p-1}.
\]
It is classical that
\[
D_n=\F_p[\Qnzero,Q_{n,1},\dots,Q_{n,n-1}]\subset \Frac(P_n)
\]
is a polynomial algebra \cite{Dickson}. We will frequently work in the localization $D_n[\Qnzero^{-1}]$.

\begin{example}[The case $n=2$]\label{ex:dickson-n2}
For $n=2$, one has
\[
L_2=[0,1]=
\det\begin{pmatrix} x_1 & x_1^p\\ x_2 & x_2^p \end{pmatrix}
=
x_1x_2^p-x_1^px_2.
\]
Moreover,
\[
L_{2,0}=[1,2],
\qquad
L_{2,1}=[0,2],
\]
so that
\[
Q_{2,0}=L_2^{\,p-1},
\qquad
Q_{2,1}=\frac{L_{2,1}}{L_2}.
\]
This is the smallest case involving a Dickson generator \(Q_{n,s}\) with \(s>0\), and it serves as a concrete model for the determinant notation.
\end{example}

\begin{remark}[On notation and independent ratio variables]\label{rem:R0-one}
For later use we set
\[
R_s:=\frac{Q_{n,s}}{Q_{n,0}} \qquad (0\le s\le n-1).
\]
Then $R_0=1$. Thus the genuinely independent normalized ratios are
\[
R_1,\dots,R_{n-1},
\]
and one has
\[
D_n[\Qnzero^{-1}]\cong \F_p[\Qnzero^{\pm 1},R_1,\dots,R_{n-1}].
\]
We keep the symbol $R_0=1$ for notational uniformity, but whenever we speak about ``ratio variables'' the independent variables are $R_1,\dots,R_{n-1}$.
\end{remark}

The action of $\StDeltai$ on Dickson generators is governed by the following:

\begin{theorem}[Sum {\cite[Theorem~2.2]{Sum}}]\label{thm:CF}
For every $s\in\{0,\dots,n-1\}$ and every $i\ge1$, one has
\[
\StDeltai(\Qns)=(-1)^n \Qnzero\Big(P_{n,i,s}^{\,p}+R_{n,i}^{\,p}\,\Qns\Big),
\]
where
\[
P_{n,i,0}:=0,
\qquad
P_{n,i,s}:=-\frac{[0,1,\dots,\widehat{s-1},\dots,n-1,i-1]}{\Ln}
\quad (s>0),
\]
and
\[
R_{n,i}:=\frac{[0,1,\dots,n-2,i-1]}{\Ln}.
\]
Here the hat means omission, and the determinant quotients above belong to $D_n$. When \(n=1\), the initial string \(0,1,\dots,n-2\) is empty, so
\[
[0,1,\dots,n-2,i-1]=[i-1].
\]
Equivalently, if one formally extends the notation by
\[
P_{n,i,n}:=-\frac{[0,1,\dots,n-2,i-1]}{\Ln},
\]
then $R_{n,i}=-P_{n,i,n}$.
\end{theorem}

\begin{example}[A concrete family from Sum]\label{ex:sum-family}
One useful special case is $i=n+1$. By \cite[Corollary~2.6]{Sum}, one has
\[
\StDeltai(\Qns)
=
(-1)^n\Qnzero\Big(-Q_{n,s-1}^{\,p}+Q_{n,n-1}^{\,p}\Qns\Big)
\qquad (0\le s<n),
\]
with the convention $Q_{n,-1}=0$. Thus, in this case,
\[
P_{n,n+1,s}=-Q_{n,s-1}
\qquad (0\le s<n),
\qquad
R_{n,n+1}=Q_{n,n-1}.
\]
In particular, the case $s=0$ gives
\[
P_{n,n+1,0}=0
\]
because $Q_{n,-1}=0$. This example is helpful to keep in mind: it shows concretely that the normalized action is affine-linear in the Dickson generators.
\end{example}

\begin{lemma}[Primitivity and the Leibniz rule]\label{lem:primitive-leibniz}
The Steenrod--Milnor operation
\[
\StDeltai=\mathrm{St}^{\emptyset,\Delta_i}
\]
is primitive in the Steenrod algebra. Equivalently,
\[
\Delta_{\mathcal A}(\StDeltai)
=
\StDeltai\otimes 1+1\otimes \StDeltai .
\]
Consequently, on $P_n$ it satisfies the ordinary Leibniz rule
\[
\StDeltai(fg)=\StDeltai(f)g+f\StDeltai(g)
\qquad (f,g\in P_n).
\]
\end{lemma}

\begin{proof}
This is a standard consequence of Milnor's dual description of the Steenrod algebra \cite{Milnor}.
Under the duality between the coproduct in the Steenrod algebra and multiplication in
\[
\mathcal A_p^*\cong E(\tau_0,\tau_1,\dots)\otimes P(\xi_1,\xi_2,\dots),
\]
the basis element dual to the pure monomial $\xi_i$ is primitive: in the monomial basis, the only factorizations of $\xi_i$ are
\[
1\cdot \xi_i
\qquad\text{and}\qquad
\xi_i\cdot 1.
\]
Thus
\[
\Delta_{\mathcal A}(\StDeltai)
=
\StDeltai\otimes 1+1\otimes\StDeltai.
\]
The displayed Leibniz rule then follows from the Cartan formula for the Steenrod action on $P_n$; see, for example, \cite{SteenrodEpstein}. Since $P_n$ is concentrated in even degrees and $\StDeltai$ has even degree, no Koszul sign occurs.
\end{proof}

\begin{remark}[Convention on localizations and fraction fields]\label{rem:localization-convention}
Whenever $\StDeltai$ is applied to an element of a localization of $D_n$, or to an element of the fraction field of such a localization, we mean the unique derivation extension of the restricted derivation
\[
\StDeltai|_{D_n}:D_n\longrightarrow D_n.
\]
Thus, for a multiplicative set $S\subset D_n$ and $a\in D_n$, $u\in S$, one has
\[
\StDeltai\left(\frac{a}{u}\right)
=
\frac{\StDeltai(a)u-a\StDeltai(u)}{u^2}.
\]
This convention is used in particular for the localizations
\[
D_n[\Qnzero^{-1}],
\qquad
D_n[\Qnzero^{-1},B^{-1}],
\]
and for their fraction fields. It should be understood as a derivation extension, not as an additional Steenrod-algebra action on arbitrary fractions.
\end{remark}

\subsection{The normalized derivation}

Theorem~\ref{thm:CF} suggests dividing out the common factor $\Qnzero$.

\begin{proposition}[The normalized operator]\label{prop:normalized-derivation}
On the localization $D_n[\Qnzero^{-1}]$, the operator
\[
\delta_i:=(-1)^n\Qnzero^{-1}\StDeltai
: D_n[\Qnzero^{-1}]\longrightarrow D_n[\Qnzero^{-1}]
\]
is a well-defined $\F_p$-linear derivation. Its values on the Dickson generators are
\[
\delta_i(\Qns)=P_{n,i,s}^{\,p}+R_{n,i}^{\,p}\,\Qns.
\]
In particular, for every $F\in D_n[\Qnzero^{-1}]$,
\[
\delta_i(F^p)=0,
\qquad
\StDeltai(F^p)=0.
\]
\end{proposition}

\begin{remark}[Restriction of the normalized derivation]\label{rem:delta-restriction}
A priori, the operator
\[
\delta_i=(-1)^n\Qnzero^{-1}\StDeltai
\]
is defined on the localization $D_n[\Qnzero^{-1}]$. However, Theorem~\ref{thm:CF} gives
\[
\delta_i(Q_{n,s})
=
P_{n,i,s}^{\,p}+R_{n,i}^{\,p}Q_{n,s}
\in D_n
\qquad (0\le s<n)
\]
for every $i$ in the range of Theorem~\ref{thm:CF}. Since
\[
D_n=\F_p[\Qnzero,Q_{n,1},\dots,Q_{n,n-1}]
\]
is generated by the Dickson generators, it follows that
\[
\delta_i(D_n)\subseteq D_n.
\]
Thus, throughout the paper, whenever no localization is explicitly needed, we regard
\[
\delta_i:D_n\longrightarrow D_n
\]
as the restricted normalized derivation. In particular, any statement concerning
$\ker(\delta_i)$ or $\im(\delta_i)$ on $D_n$ is to be understood for this restricted operator.
\end{remark}

As usual for derivations, this yields a chain rule.

\begin{remark}[Chain rule for Dickson polynomials]\label{rem:chain-rule}
Let $f\in\F_p[X_0,\dots,X_{n-1}]$ and write
\[
\underline{Q}=(\Qnzero,Q_{n,1},\dots,Q_{n,n-1}).
\]
Then
\[
\StDeltai\big(f(\underline{Q})\big)
=
(-1)^n\Qnzero\sum_{s=0}^{n-1}
\frac{\partial f}{\partial X_s}(\underline{Q})
\big(P_{n,i,s}^{\,p}+R_{n,i}^{\,p}Q_{n,s}\big).
\]
Thus the action of $\StDeltai$ on arbitrary Dickson polynomials reduces to formal differentiation.
\end{remark}

\begin{example}[A one-line computation with the chain rule]\label{ex:chain-rule-small}
Take $f(X_0,\dots,X_{n-1})=X_t^2$ for some $t$. Since
\[
\frac{\partial f}{\partial X_t}=2X_t,
\qquad
\frac{\partial f}{\partial X_s}=0\ (s\ne t),
\]
Remark~\ref{rem:chain-rule} gives
\[
\StDeltai(Q_{n,t}^2)
=
(-1)^n\Qnzero\cdot 2Q_{n,t}\big(P_{n,i,t}^{\,p}+R_{n,i}^{\,p}Q_{n,t}\big).
\]
For odd $p$ this is the expected Leibniz-rule identity
\[
\StDeltai(Q_{n,t}^2)=2Q_{n,t}\StDeltai(Q_{n,t}).
\]
\end{example}

\subsection{Iterates of the action}

\begin{theorem}[Closed form for iterates of $\delta_i$ and $\StDeltai$]\label{thm:iterate-St}
Fix $n,p$ and $i$. For $s\in\{0,1,\dots,n-1\}$ set
\[
A_s:=P_{n,i,s}^{\,p},
\qquad
B:=R_{n,i}^{\,p}.
\]
Then, for every integer $m\ge 1$,
\begin{equation}\label{eq:B1}
\delta_i^{\,m}(Q_{n,s})
=
B^{\,m}Q_{n,s}+B^{\,m-1}A_s.
\end{equation}
and
\begin{equation}\label{eq:B2}
(\StDeltai)^{\,m}(Q_{n,s})
=
(-1)^{mn}m!\,\Qnzero^{\,m}\Big(B^{\,m}Q_{n,s}+B^{\,m-1}A_s\Big).
\end{equation}
Here $m!$ is viewed as an element of the base field $\F_p$. It arises from the identity
\[
\sum_{k=0}^m \begin{bmatrix} m \\ k \end{bmatrix}=m!,
\]
where $\begin{bmatrix} m \\ k \end{bmatrix}$ denotes the unsigned Stirling numbers of the first kind.
\end{theorem}

\begin{corollary}\label{cor:iterate-ideal}
For all $m\ge 1$ and all $s$,
\[
(\StDeltai)^{\,m}(Q_{n,s})\in (\Qnzero^{\,m}).
\]
Moreover,
\[
(\StDeltai)^{\,m}(Q_{n,s})=0
\qquad \text{for all } m\ge p.
\]
In fact, the operator identity
\[
(\StDeltai)^p=0
\qquad \text{on } D_n
\]
holds. Consequently $(\StDeltai)^m=0$ on $D_n$ for every $m\ge p$.
\end{corollary}

\begin{proof}
The ideal-containment statement and the vanishing on Dickson generators follow immediately from Theorem~\ref{thm:iterate-St}. To prove the operator identity on $D_n$, set
\[
D:=\StDeltai.
\]
The operation $D$ acts as an ordinary derivation on $D_n$. Since the characteristic is $p$, the $p$-fold iterate $D^p$ is again a derivation: for $f,g\in D_n$,
\[
D^p(fg)=\sum_{k=0}^{p}\binom pk D^k(f)D^{p-k}(g)
=D^p(f)g+fD^p(g).
\]
By the generator formula above, $D^p(Q_{n,s})=0$ for every Dickson generator $Q_{n,s}$, $0\le s<n$. Since $D_n$ is the polynomial algebra generated by these elements, the derivation $D^p$ is identically zero on $D_n$.
\end{proof}

\begin{remark}[Ambient \(p\)-nilpotence]\label{rem:ambient-p-nilpotence}
The operator identity in Corollary~\ref{cor:iterate-ideal} is compatible with the standard action on the ambient polynomial algebra. Indeed,
\[
\StDeltai(x_r)=x_r^{p^i}
\qquad (1\le r\le n);
\]
see \cite{Sum2007,SteenrodEpstein}. Since \(\StDeltai\) is a derivation, it annihilates the \(p\)-th power \(x_r^{p^i}\). Moreover, \((\StDeltai)^p\) is a derivation in characteristic \(p\). It vanishes on every polynomial generator \(x_r\), and hence
\[
(\StDeltai)^p=0
\qquad\text{on }P_n.
\]
The statement on \(D_n\) is therefore also the restriction of this ambient identity. The content of Theorem~\ref{thm:iterate-St} is finer: it determines every intermediate iterate on each Dickson generator and records its precise divisibility by \(Q_{n,0}^m\).
\end{remark}

\begin{proposition}[Restricted \(p\)-power relation]\label{prop:delta-p-relation}
With
\[
B=R_{n,i}^{\,p},
\]
the normalized derivation satisfies
\begin{equation}\label{eq:delta-p-relation}
\delta_i^{\,p}=B^{p-1}\delta_i
\end{equation}
on \(D_n\) and on \(D_n[Q_{n,0}^{-1}]\). Consequently, for every integer \(m\ge1\),
\begin{equation}\label{eq:delta-periodicity}
\delta_i^{\,m+p-1}=B^{p-1}\delta_i^{\,m}.
\end{equation}
In particular, \(B=0\) implies \(\delta_i^p=0\), whereas \(B=1\) implies \(\delta_i^p=\delta_i\).
\end{proposition}

\begin{example}[The third iterate]\label{ex:efficiency}
For \(m=3\), Theorem~\ref{thm:iterate-St} gives directly
\[
(\StDeltai)^3(Q_{n,s})
=
(-1)^n6Q_{n,0}^{\,3}
\big(B^3Q_{n,s}+B^2A_s\big).
\]
This expression is zero when \(p=3\), while for \(p>3\) it gives the third iterate without expanding successive Leibniz rules. The Stirling-number identity in Proposition~\ref{prop:stirling-identity} explains the coefficient \(3!=6\).
\end{example}

\subsection{Kernel and image constructions}

\begin{proposition}[A kernel family]\label{prop:C-kernel}
For every $F(\underline{Q})\in D_n$ and every $s\in\{0,\dots,n-1\}$,
\[
\StDeltai\big(\Qnzero^{\,p-1}F(\underline{Q})^{\,p}Q_{n,s}\big)
=
(-1)^n\Qnzero^{\,p}\big(F(\underline{Q})P_{n,i,s}\big)^p.
\]
In particular, the element on the right-hand side belongs to $\ker(\StDeltai)$.
\end{proposition}

\begin{remark}
Proposition~\ref{prop:C-kernel} is a direct generalization of a result of Sum \SumCite. Indeed, when $F(\underline{Q})=1$, it recovers Corollary~2.8 of \cite{Sum}.
\end{remark}

\begin{corollary}[Global properties of the action]\label{cor:global-props}
The image of \(\StDeltai\) is contained in the principal ideal generated by \(Q_{n,0}\):
\[
\im(\StDeltai)\subseteq (Q_{n,0}).
\]
Moreover, \(\StDeltai\) annihilates every Frobenius power:
\[
\StDeltai(H^p)=0
\qquad (H\in D_n).
\]
Equivalently, if \(D_n^{(p)}\) denotes the Frobenius image of \(D_n\), then
\[
\StDeltai(D_n^{(p)})=0.
\]
In particular,
\[
\StDeltai(Q_{n,0}^{\,p}G^p)=0
\qquad (G\in D_n),
\]
so \(\StDeltai\) vanishes on the \(D_n^{(p)}\)-submodule
\[
Q_{n,0}^{\,p}D_n^{(p)}\subset D_n.
\]
The latter submodule is not asserted to be an ideal of \(D_n\).
\end{corollary}

\subsection{Normalized ratios}

\begin{proposition}[Invariant ratios]\label{prop:invariant-ratios}
Work in the localization $D_n[\Qnzero^{-1}]$. Let
\[
A_s:=P_{n,i,s}^{\,p},
\qquad
B:=R_{n,i}^{\,p}.
\]
Assume $B\ne0$, and work further in the localization
\[
D_n[\Qnzero^{-1},B^{-1}].
\]
Then for each $0\le s<n$ the element
\[
I_s:=\frac{Q_{n,s}+B^{-1}A_s}{\Qnzero}
\]
belongs to the kernel of the normalized derivation, that is,
\[
\delta_i(I_s)=0.
\]
Moreover,
\[
I_0=1.
\]
Thus the nontrivial invariant ratios in this family are
\[
I_1,\dots,I_{n-1}.
\]
Consequently, every $\F_p$-polynomial in $I_1,\dots,I_{n-1}$ is annihilated by $\delta_i$.
More generally, every rational function over $\F_p$ in $I_1,\dots,I_{n-1}$ whose evaluation is defined in the fraction field of the chosen localization is annihilated by $\delta_i$ as well, and hence also by $\StDeltai$.
\end{proposition}

\begin{theorem}[Action on normalized ratios]\label{thm:ratio-action}
Work in $D_n[\Qnzero^{-1}]$ and define
\[
R_s:=\frac{Q_{n,s}}{Q_{n,0}}
\qquad (0\le s\le n-1),
\]
with $R_0=1$ as in Remark~\ref{rem:R0-one}. Let
\[
A_s:=P_{n,i,s}^{\,p},
\qquad
K:=\Frac\big(D_n[\Qnzero^{-1}]\big).
\]
Then for every $0\le s<n$ one has
\begin{equation}\label{eq:C1}
\StDeltai(R_s)=(-1)^nA_s.
\end{equation}
In particular, for every polynomial
\[
\Phi\in \F_p[T_1,\dots,T_{n-1}],
\]
one has
\begin{equation}\label{eq:C2}
\StDeltai\big(\Phi(R_1,\dots,R_{n-1})\big)
=
(-1)^n\sum_{s=1}^{n-1}
\frac{\partial \Phi}{\partial T_s}(R_1,\dots,R_{n-1})\,A_s
\end{equation}
as an identity in $D_n[\Qnzero^{-1}]$.

More generally, the same formula holds in $K$ for every rational function
\[
\Phi\in \F_p(T_1,\dots,T_{n-1})
\]
whose evaluation at $(R_1,\dots,R_{n-1})$ is defined in $K$.
\end{theorem}

\begin{remark}[Normalized ratios and invariant ratios]\label{rem:ratio-reading}
Formula \eqref{eq:C2} is first an identity in the localization $D_n[\Qnzero^{-1}]$ for polynomial expressions in the independent ratios $R_1,\dots,R_{n-1}$. Its rational-function version is understood in the fraction field $\Frac\big(D_n[\Qnzero^{-1}]\big)$.
It does not assert that all coefficients $A_s$ necessarily belong to the polynomial ring $\F_p[R_1,\dots,R_{n-1}]$. Example~\ref{ex:sum-family} shows that, for $i=n+1$, one can have $A_s=-Q_{n,s-1}^{\,p}$,
which involves $Q_{n,0}$ explicitly. Thus the normalized ratios provide a useful affine-linear description of the action, but one must distinguish carefully between the ambient localization, its fraction field, and the smaller ratio subring.

When $B\ne0$ and after localizing by inverting $B$, the invariant ratios of Proposition~\ref{prop:invariant-ratios} are obtained from the normalized ratios by the correction term
\[
I_s=R_s+\frac{B^{-1}A_s}{Q_{n,0}}.
\]
For $s=0$, since $R_0=1$ and $A_0=0$, this gives
\[
I_0=1.
\]
Thus the genuinely useful invariant-ratio variables are $I_1,\dots,I_{n-1}$.

Since
\[
\delta_i(R_s)=\frac{A_s}{Q_{n,0}},
\qquad
\delta_i\!\left(\frac{B^{-1}A_s}{Q_{n,0}}\right)=-\frac{A_s}{Q_{n,0}},
\]
we recover directly that
\[
\delta_i(I_s)=0,
\]
and therefore also
\[
\StDeltai(I_s)=0
\]
in the same localization.
\end{remark}

\subsection{An Euler-type family after affine normalization}

\begin{theorem}[Euler-type family]\label{thm:euler-family}
Let
\[
B:=R_{n,i}^{\,p}.
\]
Assume $B\ne0$, work in the localization $D_n[\Qnzero^{-1},B^{-1}]$, and let $L$ denote its fraction field. Let $I_1,\dots,I_{n-1}$ be the invariant ratios of Proposition~\ref{prop:invariant-ratios}. Then, for every integer $m\in\mathbb Z$ and every rational function
\[
\Phi\in \F_p(T_1,\dots,T_{n-1})
\]
whose evaluation at $(I_1,\dots,I_{n-1})$ is defined in $L$, one has
\begin{equation}\label{eq:euler-delta}
\delta_i\big(\Qnzero^m\Phi(I_1,\dots,I_{n-1})\big)
=
mB\,\Qnzero^m\Phi(I_1,\dots,I_{n-1}),
\end{equation}
and therefore
\begin{equation}\label{eq:euler-st}
\StDeltai\big(\Qnzero^m\Phi(I_1,\dots,I_{n-1})\big)
=
(-1)^n mB\,\Qnzero^{m+1}\Phi(I_1,\dots,I_{n-1}).
\end{equation}
In particular, if $p\mid m$, then
\[
\Qnzero^m\Phi(I_1,\dots,I_{n-1})\in \ker(\delta_i)\cap\ker(\StDeltai).
\]
If $p\nmid m$, then
\[
\Qnzero^m\Phi(I_1,\dots,I_{n-1})\in \im(\delta_i)
\]
in $L$.
\end{theorem}

\begin{remark}
Theorem~\ref{thm:euler-family} shows that, when \(B\ne0\) and after localizing by inverting \(B\), the affine change of variables
\[
U_s:=Q_{n,s}+B^{-1}A_s
\qquad (0\le s\le n-1)
\]
satisfies
\[
\delta_i(U_s)=BU_s,
\]
while the normalized ratios \(I_s=U_s/Q_{n,0}\) are invariant. On the family
\[
Q_{n,0}^m\Phi(I_1,\dots,I_{n-1}),
\]
the normalized derivation therefore acts as an Euler operator with eigenvalue \(mB\).
\end{remark}

\begin{theorem}[Finite-field weight decomposition]\label{thm:weight-decomposition}
Assume
\[
B=R_{n,i}^{\,p}\ne0
\]
and set
\[
S_i:=D_n[Q_{n,0}^{-1},B^{-1}],
\qquad
\partial_i:=B^{-1}\delta_i.
\]
Then
\[
\partial_i^p=\partial_i
\]
as an \(\F_p\)-linear derivation of \(S_i\). For \(\lambda\in\F_p\), define
\[
S_i^{(\lambda)}
:=
\ker(\partial_i-\lambda\operatorname{id}_{S_i}).
\]
There is a direct-sum decomposition
\begin{equation}\label{eq:weight-decomposition}
S_i=\bigoplus_{\lambda\in\F_p}S_i^{(\lambda)}.
\end{equation}
The projection onto \(S_i^{(\lambda)}\) is the Lagrange polynomial in \(\partial_i\)
\begin{equation}\label{eq:weight-projector}
\pi_\lambda
=
\prod_{\substack{\mu\in\F_p\\ \mu\ne\lambda}}
\frac{\partial_i-\mu\operatorname{id}_{S_i}}{\lambda-\mu}.
\end{equation}
The decomposition is multiplicative:
\begin{equation}\label{eq:weight-product}
S_i^{(\lambda)}S_i^{(\mu)}
\subseteq
S_i^{(\lambda+\mu)}
\qquad (\lambda,\mu\in\F_p),
\end{equation}
where the sum of weights is taken in \(\F_p\). In particular,
\[
Q_{n,0}\in S_i^{(1)},
\qquad
B\in S_i^{(0)},
\qquad
I_s\in S_i^{(0)}
\quad (1\le s<n).
\]

For \(f\in S_i^{(\lambda)}\) and \(m\ge1\), one has
\begin{equation}\label{eq:weight-iterate}
(\StDeltai)^m(f)
=
(-1)^{mn}B^mQ_{n,0}^{\,m}
\left(\prod_{r=0}^{m-1}(\lambda+r)\right)f,
\end{equation}
with every factor \(\lambda+r\) interpreted in \(\F_p\). Consequently, as \(\F_p\)-vector spaces,
\begin{equation}\label{eq:weight-kernel-image-delta}
\ker(\delta_i)=S_i^{(0)},
\qquad
\im(\delta_i)=\bigoplus_{\lambda\in\F_p^\times}S_i^{(\lambda)},
\end{equation}
and
\begin{equation}\label{eq:weight-kernel-image-st}
\ker(\StDeltai)=S_i^{(0)},
\qquad
\im(\StDeltai)
=
\bigoplus_{\lambda\in\F_p\setminus\{1\}}S_i^{(\lambda)}.
\end{equation}
\end{theorem}

\begin{remark}\label{rem:euler-as-weight}
Combining Theorem~\ref{thm:euler-family} and Theorem~\ref{thm:weight-decomposition}, let \(\Phi\) be a polynomial in the invariant ratios, or more generally a rational function whose evaluation \(\Phi(I_1,\dots,I_{n-1})\) belongs to \(S_i\). Then
\[
Q_{n,0}^m\Phi(I_1,\dots,I_{n-1})
\in
S_i^{(\overline m)},
\]
where \(\overline m\) is the class of \(m\) in \(\F_p\). Thus the Euler-type family inside \(S_i\) is generated by the weight-one unit \(Q_{n,0}\) and the weight-zero invariant ratios.
\end{remark}

\section{Proofs of the main results}\label{sec:proofs}

This section proves the structural statements of Section~\ref{sec:main}. The argument begins with the derivation extension and the affine differential system on the Dickson generators. It then resolves the noncommuting product \(Q_{n,0}\delta_i\) through unsigned Stirling numbers, establishes the restricted \(p\)-power relation, and concludes with the ratio, Euler, and weight-space calculations.

\subsection*{Proof of Proposition~\ref{prop:normalized-derivation}}
By Lemma~\ref{lem:primitive-leibniz}, the operation
\[
\StDeltai=\mathrm{St}^{\emptyset,\Delta_i}
\]
acts as an ordinary $\F_p$-linear derivation on $P_n$. Sum's formula shows that $\StDeltai$ stabilizes the Dickson algebra $D_n$, since it gives an explicit element of $D_n$ for each value $\StDeltai(Q_{n,s})$. Therefore $\StDeltai|_{D_n}$ extends uniquely to a derivation on the localization $D_n[\Qnzero^{-1}]$ by the quotient rule, as explained in Remark~\ref{rem:localization-convention}.

For $f,g\in D_n[\Qnzero^{-1}]$, we compute
\[
\delta_i(fg)
=
(-1)^n\Qnzero^{-1}\StDeltai(fg)
=
(-1)^n\Qnzero^{-1}\big(\StDeltai(f)g+f\StDeltai(g)\big).
\]
Hence
\[
\delta_i(fg)
=
\delta_i(f)g+f\delta_i(g),
\]
so $\delta_i$ is a derivation on $D_n[\Qnzero^{-1}]$.

The formula
\[
\delta_i(Q_{n,s})=P_{n,i,s}^{\,p}+R_{n,i}^{\,p}Q_{n,s}
\]
follows immediately from Theorem~\ref{thm:CF}. Finally, in characteristic $p$, every derivation annihilates $p$-th powers. Therefore, for every $F\in D_n[\Qnzero^{-1}]$,
\[
\delta_i(F^p)=pF^{p-1}\delta_i(F)=0.
\]
Since
\[
\StDeltai=(-1)^n\Qnzero\,\delta_i
\]
on the localization, it follows also that
\[
\StDeltai(F^p)=0.
\]
\qed

\subsection*{Proof of Theorem~\ref{thm:iterate-St}}
By Theorem~\ref{thm:CF} and Proposition~\ref{prop:normalized-derivation},
\[
\delta_i(Q_{n,s})=A_s+BQ_{n,s},
\qquad
\delta_i(\Qnzero)=B\Qnzero.
\]
Since $A_s$ and $B$ are $p$-th powers, one has
\[
\delta_i(A_s)=0,
\qquad
\delta_i(B)=0.
\]
We first prove \eqref{eq:B1} by induction on $m$. The case $m=1$ is immediate. Assume
\[
\delta_i^k(Q_{n,s})=B^kQ_{n,s}+B^{k-1}A_s.
\]
Then
\begin{align*}
\delta_i^{k+1}(Q_{n,s})
&=
\delta_i\big(B^kQ_{n,s}+B^{k-1}A_s\big)\\
&=
B^k\delta_i(Q_{n,s})+\delta_i(B^{k-1}A_s)\\
&=
B^k(A_s+BQ_{n,s})+0\\
&=
B^{k+1}Q_{n,s}+B^kA_s.
\end{align*}
This proves \eqref{eq:B1}.

To prove \eqref{eq:B2}, let $M$ be the operator of multiplication by $\Qnzero$. Then
\[
\StDeltai=(-1)^nM\delta_i.
\]
The operators $M$ and $\delta_i$ do not commute, but they satisfy the relation
\[
\delta_iM=M(\delta_i+B).
\]
More generally, one has the following commuting rule.

\begin{lemma}\label{lem:commutation-general}
For every integer \(m\ge1\),
\[
\delta_iM^m=M^m(\delta_i+mB).
\]
\end{lemma}

\begin{proof}
Let \(f\in D_n[Q_{n,0}^{-1}]\). Since \(M^m(f)=Q_{n,0}^mf\) and
\[
\delta_i(Q_{n,0})=BQ_{n,0},
\qquad
\delta_i(B)=0,
\]
the Leibniz rule gives
\[
\delta_i(Q_{n,0}^m)
=
mQ_{n,0}^{m-1}\delta_i(Q_{n,0})
=
mBQ_{n,0}^m.
\]
Consequently,
\begin{align*}
\delta_iM^m(f)
&=
\delta_i(Q_{n,0}^mf)\\
&=
\delta_i(Q_{n,0}^m)f+Q_{n,0}^m\delta_i(f)\\
&=
Q_{n,0}^m\big(\delta_i(f)+mBf\big)\\
&=
M^m(\delta_i+mB)(f).
\end{align*}
Since this identity holds for every \(f\), the asserted operator relation follows.
\end{proof}

We next expand the noncommuting power \((M\delta_i)^m\).

\begin{proposition}\label{prop:stirling-identity}
For every integer \(m\ge1\),
\[
(M\delta_i)^m
=
M^m\sum_{j=1}^{m}
\begin{bmatrix} m \\ j \end{bmatrix}
B^{m-j}\delta_i^j,
\]
where \(\begin{bmatrix}m\\j\end{bmatrix}\) is the unsigned Stirling number of the first kind.
\end{proposition}

\begin{proof}
We use induction on \(m\). For \(m=1\), the identity is
\[
M\delta_i
=
M\begin{bmatrix}1\\1\end{bmatrix}\delta_i,
\]
which holds because \(\begin{bmatrix}1\\1\end{bmatrix}=1\).

Assume that the formula holds for a fixed \(m\ge1\). Then
\begin{align*}
(M\delta_i)^{m+1}
&=
M\delta_i(M\delta_i)^m\\
&=
M\delta_i
\left(
M^m
\sum_{j=1}^{m}
\begin{bmatrix}m\\j\end{bmatrix}
B^{m-j}\delta_i^j
\right)\\
&=
M^{m+1}(\delta_i+mB)
\sum_{j=1}^{m}
\begin{bmatrix}m\\j\end{bmatrix}
B^{m-j}\delta_i^j,
\end{align*}
where Lemma~\ref{lem:commutation-general} was used in the last equality. Since \(\delta_i(B)=0\), the derivation \(\delta_i\) commutes with multiplication by every power of \(B\). Therefore
\begin{align*}
(M\delta_i)^{m+1}
=
M^{m+1}
\left(
\sum_{j=1}^{m}
\begin{bmatrix}m\\j\end{bmatrix}
B^{m-j}\delta_i^{j+1}
+
\sum_{j=1}^{m}
m\begin{bmatrix}m\\j\end{bmatrix}
B^{m+1-j}\delta_i^j
\right).
\end{align*}
In the first sum put \(k=j+1\); in the second put \(k=j\). With the conventions
\[
\begin{bmatrix}m\\0\end{bmatrix}=0,
\qquad
\begin{bmatrix}m\\m+1\end{bmatrix}=0,
\]
the coefficient of \(B^{m+1-k}\delta_i^k\), for \(1\le k\le m+1\), is
\[
\begin{bmatrix}m\\k-1\end{bmatrix}
+
m\begin{bmatrix}m\\k\end{bmatrix}.
\]
The defining recurrence for unsigned Stirling numbers of the first kind gives
\[
\begin{bmatrix}m+1\\k\end{bmatrix}
=
\begin{bmatrix}m\\k-1\end{bmatrix}
+
m\begin{bmatrix}m\\k\end{bmatrix}.
\]
Hence
\[
(M\delta_i)^{m+1}
=
M^{m+1}
\sum_{k=1}^{m+1}
\begin{bmatrix}m+1\\k\end{bmatrix}
B^{m+1-k}\delta_i^k,
\]
which is the required formula for \(m+1\).
\end{proof}

We now apply these results to find $(\StDeltai)^m(Q_{n,s})$. Indeed, we have
\begin{align*}
(\StDeltai)^m(Q_{n,s}) &= (-1)^{mn} (\Qnzero \delta_i)^m (Q_{n,s}) \\
&= (-1)^{mn} \left( \Qnzero^m \sum_{j=1}^{m} \begin{bmatrix} m \\ j \end{bmatrix}  B^{m-j} \delta_i^j \right) (Q_{n,s}) && \text{(by Proposition~\ref{prop:stirling-identity})} \\
&= (-1)^{mn} \Qnzero^m \sum_{j=1}^{m} \begin{bmatrix} m \\ j \end{bmatrix} B^{m-j} (\delta_i^j(Q_{n,s})).
\end{align*}
We substitute the result for $\delta_i^j(Q_{n,s})$ from \eqref{eq:B1} and obtain:
\begin{align*}
(\StDeltai)^m(Q_{n,s}) &= (-1)^{mn} \Qnzero^m \sum_{j=1}^{m} \begin{bmatrix} m \\ j \end{bmatrix} (B^m Q_{n,s} + B^{m-1} A_s) \\
&= (-1)^{mn} \Qnzero^m \sum_{j=0}^{m} \begin{bmatrix} m \\ j \end{bmatrix} (B^m Q_{n,s} + B^{m-1} A_s)&&\text{(since $\begin{bmatrix} m \\ 0 \end{bmatrix} = 0$)} \\
&= (-1)^{mn} m! \Qnzero^m (B^m Q_{n,s} + B^{m-1} A_s)
&& \text{(in $\F_p$, since $\sum_{j=0}^m \begin{bmatrix} m \\ j \end{bmatrix} = m!$)}.
\end{align*}
This completes the proof of \eqref{eq:B2}.
\qed

\begin{remark}
While formula \eqref{eq:B2} can also be established by a direct, albeit \textit{computationally complicated}, induction on $m$, the proof presented above is chosen for conceptual clarity. Our operator-theoretic approach is designed to elucidate the combinatorial structure underlying the iteration, explaining precisely how the non-commutativity of operators gives rise to the factorial term $m!$ via Stirling numbers.
\end{remark}

\subsection*{Proof of Proposition~\ref{prop:delta-p-relation}}

The \(p\)-fold iterate of any derivation in characteristic \(p\) is again a derivation. Indeed, for \(f,g\in D_n\), the generalized Leibniz formula gives
\[
\delta_i^p(fg)
=
\sum_{k=0}^{p}\binom pk
\delta_i^k(f)\delta_i^{p-k}(g).
\]
Every intermediate binomial coefficient \(\binom pk\), \(1\le k\le p-1\), vanishes in \(\F_p\). Hence
\[
\delta_i^p(fg)
=
\delta_i^p(f)g+f\delta_i^p(g).
\]
The operator \(B^{p-1}\delta_i\) is also a derivation, because multiplication by the central element \(B^{p-1}\) preserves the Leibniz rule.

For each Dickson generator \(Q_{n,s}\), formula \eqref{eq:B1} with \(m=p\) gives
\[
\delta_i^p(Q_{n,s})
=
B^pQ_{n,s}+B^{p-1}A_s
=
B^{p-1}(BQ_{n,s}+A_s)
=
B^{p-1}\delta_i(Q_{n,s}).
\]
Thus the two derivations \(\delta_i^p\) and \(B^{p-1}\delta_i\) agree on the algebraically independent generators
\[
Q_{n,0},Q_{n,1},\dots,Q_{n,n-1}
\]
of \(D_n\). They therefore agree on every element of \(D_n\), proving
\[
\delta_i^p=B^{p-1}\delta_i
\qquad\text{on }D_n.
\]
Both sides extend uniquely as derivations to \(D_n[Q_{n,0}^{-1}]\), so the same identity holds on the localization.

Since \(\delta_i(B)=0\), the derivation \(\delta_i\) commutes with multiplication by \(B^{p-1}\). For \(m\ge1\), composing \eqref{eq:delta-p-relation} with \(\delta_i^{m-1}\) therefore yields
\[
\delta_i^{m+p-1}
=
\delta_i^{m-1}\delta_i^p
=
\delta_i^{m-1}B^{p-1}\delta_i
=
B^{p-1}\delta_i^m.
\]
The final two assertions follow by substituting \(B=0\) and \(B=1\), respectively.
\qed

\subsection*{Proof of Proposition~\ref{prop:C-kernel}}
If $s=0$, then
\[
\Qnzero^{p-1}F(\underline Q)^pQ_{n,0}=(\Qnzero F(\underline Q))^p,
\]
so the left-hand side is zero by Proposition~\ref{prop:normalized-derivation}; the right-hand side is also zero because $P_{n,i,0}=0$. We may therefore assume $s>0$.

Apply Remark~\ref{rem:chain-rule} to
\[
f(X_0,\dots,X_{n-1})=X_0^{p-1}\big(F(X_0,\dots,X_{n-1})\big)^pX_s.
\]
Because $\partial(F^p)/\partial X_t=0$ in characteristic $p$, only the derivatives with respect to the distinct variables $X_0$ and $X_s$ survive. Hence
\begin{align*}
\StDeltai\big(\Qnzero^{p-1}F(\underline{Q})^pQ_{n,s}\big)
&=
(p-1)\Qnzero^{p-2}F(\underline{Q})^pQ_{n,s}\,\StDeltai(\Qnzero)\\
&\qquad +\Qnzero^{p-1}F(\underline{Q})^p\,\StDeltai(Q_{n,s}).
\end{align*}
Using Theorem~\ref{thm:CF} and the congruence $(p-1)\equiv -1\pmod p$, we get
\begin{align*}
\StDeltai\big(\Qnzero^{p-1}F^pQ_{n,s}\big)
&=
(-1)^n\Big[-\Qnzero^{p-2}F^pQ_{n,s}(B\Qnzero^2)
+\Qnzero^{p-1}F^p\Qnzero(A_s+BQ_{n,s})\Big]\\
&=
(-1)^n\Qnzero^pF^pA_s\\
&=
(-1)^n\Qnzero^p(FP_{n,i,s})^p.
\end{align*}
Finally, Proposition~\ref{prop:normalized-derivation} implies that $\StDeltai$ annihilates all $p$-th powers, so the resulting element belongs to $\ker(\StDeltai)$.
\qed

\subsection*{Proof of Corollary~\ref{cor:global-props}}

By Remark~\ref{rem:delta-restriction},
\[
\delta_i(D_n)\subseteq D_n,
\]
and on \(D_n\) one has
\[
\StDeltai=(-1)^nQ_{n,0}\delta_i.
\]
It follows immediately that every value of \(\StDeltai\) is divisible by \(Q_{n,0}\), so
\[
\im(\StDeltai)\subseteq(Q_{n,0}).
\]

Let \(H\in D_n\). Since \(\delta_i\) is a derivation and the characteristic is \(p\),
\[
\delta_i(H^p)=pH^{p-1}\delta_i(H)=0.
\]
Therefore
\[
\StDeltai(H^p)
=
(-1)^nQ_{n,0}\delta_i(H^p)
=
0.
\]
This proves
\[
\StDeltai(D_n^{(p)})=0.
\]
Finally, for \(G\in D_n\),
\[
Q_{n,0}^{\,p}G^p=(Q_{n,0}G)^p\in D_n^{(p)},
\]
and hence \(\StDeltai\) also vanishes on \(Q_{n,0}^{\,p}D_n^{(p)}\).
\qed

\subsection*{Proof of Proposition~\ref{prop:invariant-ratios}}
Since $B\ne0$, the localization $D_n[\Qnzero^{-1},B^{-1}]$ is defined. Set
\[
U_s:=Q_{n,s}+B^{-1}A_s.
\]
Since $A_s$ and $B$ are $p$-th powers, $\delta_i(A_s)=\delta_i(B)=0$, and therefore $\delta_i(B^{-1}A_s)=0$. Using Proposition~\ref{prop:normalized-derivation},
\[
\delta_i(U_s)=\delta_i(Q_{n,s})=A_s+BQ_{n,s}=B(Q_{n,s}+B^{-1}A_s)=BU_s.
\]
Now apply the quotient rule:
\[
\delta_i(I_s)
=
\delta_i\left(\frac{U_s}{\Qnzero}\right)
=
\frac{\delta_i(U_s)\Qnzero-U_s\delta_i(\Qnzero)}{\Qnzero^2}
=
\frac{BU_s\Qnzero-U_sB\Qnzero}{\Qnzero^2}=0.
\]

For $s=0$, we have $A_0=P_{n,i,0}^{\,p}=0$. Hence
\[
I_0
=
\frac{Q_{n,0}+B^{-1}A_0}{\Qnzero}
=
\frac{\Qnzero}{\Qnzero}
=
1.
\]
Thus only $I_1,\dots,I_{n-1}$ give nontrivial invariant ratios.

The final claim follows from the fact that $\delta_i$ is a derivation, hence extends uniquely
to the fraction field of the chosen localization, where it annihilates every rational function
in $I_1,\dots,I_{n-1}$ by repeated use of the quotient rule.
Since
\[
\StDeltai=(-1)^n\Qnzero\,\delta_i
\]
on that localization, the same elements are also annihilated by $\StDeltai$.
\qed

\subsection*{Proof of Theorem~\ref{thm:ratio-action}}
From Proposition~\ref{prop:normalized-derivation} we have
\[
\delta_i(Q_{n,s})=A_s+BQ_{n,s},
\qquad
\delta_i(\Qnzero)=B\Qnzero.
\]
Therefore
\begin{align*}
\delta_i(R_s)
&=
\delta_i\left(\frac{Q_{n,s}}{Q_{n,0}}\right)\\
&=
\frac{(A_s+BQ_{n,s})Q_{n,0}-Q_{n,s}(BQ_{n,0})}{Q_{n,0}^2}\\
&=
\frac{A_s}{Q_{n,0}}.
\end{align*}
Multiplying by $(-1)^nQ_{n,0}$ yields
\[
\StDeltai(R_s)=(-1)^nA_s,
\]
which proves \eqref{eq:C1}.

Now let
\[
K:=\Frac\big(D_n[\Qnzero^{-1}]\big).
\]
Since $\delta_i$ is a derivation on $D_n[\Qnzero^{-1}]$, it extends uniquely to a derivation on $K$.

For a polynomial $\Phi\in \F_p[T_1,\dots,T_{n-1}]$, the ordinary chain rule gives
\[
\delta_i\big(\Phi(R_1,\dots,R_{n-1})\big)
=
\sum_{s=1}^{n-1}
\frac{\partial \Phi}{\partial T_s}(R_1,\dots,R_{n-1})\,\delta_i(R_s).
\]
Substituting $\delta_i(R_s)=A_s/Q_{n,0}$ and multiplying by $(-1)^nQ_{n,0}$, we obtain \eqref{eq:C2}.

For a rational function $\Phi\in \F_p(T_1,\dots,T_{n-1})$ whose evaluation at $(R_1,\dots,R_{n-1})$ is defined in $K$, write
\[
\Phi=\frac{f}{g}
\]
with $f,g\in \F_p[T_1,\dots,T_{n-1}]$ and $g(R_1,\dots,R_{n-1})\neq 0$ in $K$.
Applying the quotient rule in $K$ to
\[
\frac{f(R_1,\dots,R_{n-1})}{g(R_1,\dots,R_{n-1})}
\]
shows that the same formula holds for $\Phi$.
\qed

\subsection*{Proof of Theorem~\ref{thm:euler-family}}

Let $L$ be the fraction field of $D_n[\Qnzero^{-1},B^{-1}]$.
By Proposition~\ref{prop:invariant-ratios}, one has
\[
\delta_i(I_s)=0
\qquad (1\le s\le n-1).
\]
Hence, for every rational function
\[
\Phi\in \F_p(T_1,\dots,T_{n-1})
\]
whose evaluation at $(I_1,\dots,I_{n-1})$ is defined in $L$, repeated use of the quotient rule gives
\[
\delta_i\big(\Phi(I_1,\dots,I_{n-1})\big)=0.
\]
On the other hand, Proposition~\ref{prop:normalized-derivation} gives
\[
\delta_i(\Qnzero)=B\Qnzero,
\qquad
\delta_i(B)=0.
\]
Therefore, for $m\ge 0$, a straightforward induction yields
\[
\delta_i(\Qnzero^m)=mB\Qnzero^m.
\]
If $m<0$, write $m=-r$ with $r>0$. Then the quotient rule gives
\[
0=\delta_i(1)=\delta_i(\Qnzero^r\Qnzero^{-r})
=rB\Qnzero^r\Qnzero^{-r}+\Qnzero^r\delta_i(\Qnzero^{-r}),
\]
so
\[
\delta_i(\Qnzero^{-r})=-rB\Qnzero^{-r}=mB\Qnzero^m.
\]
Thus the formula
\[
\delta_i(\Qnzero^m)=mB\Qnzero^m
\]
holds for every integer $m\in\mathbb Z$.

Applying the Leibniz rule, we obtain
\[
\delta_i\big(\Qnzero^m\Phi(I_1,\dots,I_{n-1})\big)
=
\delta_i(\Qnzero^m)\Phi(I_1,\dots,I_{n-1})
+
\Qnzero^m\delta_i\big(\Phi(I_1,\dots,I_{n-1})\big),
\]
and hence
\[
\delta_i\big(\Qnzero^m\Phi(I_1,\dots,I_{n-1})\big)
=
mB\,\Qnzero^m\Phi(I_1,\dots,I_{n-1}),
\]
which is \eqref{eq:euler-delta}. Multiplying by $(-1)^n\Qnzero$ gives \eqref{eq:euler-st}.

If $p\mid m$, then the scalar $m$ vanishes in $\F_p$, so both formulas show that
\[
\Qnzero^m\Phi(I_1,\dots,I_{n-1})\in \ker(\delta_i)\cap\ker(\StDeltai).
\]
If $p\nmid m$, then $m\in\F_p^{\times}$ and $B$ is invertible in $L$, so $mB$ is invertible in $L$. Since $\delta_i(B)=0$, we have
\[
\delta_i\big((mB)^{-1}\Qnzero^m\Phi(I_1,\dots,I_{n-1})\big)
=
\Qnzero^m\Phi(I_1,\dots,I_{n-1}),
\]
which proves the image statement.
\qed

\subsection*{Proof of Theorem~\ref{thm:weight-decomposition}}

Because \(B\) is a \(p\)-th power, Proposition~\ref{prop:normalized-derivation} gives
\[
\delta_i(B)=0.
\]
It follows that \(\delta_i(B^{-1})=0\) in \(S_i\), and hence multiplication by \(B^{-1}\) commutes with \(\delta_i\). Therefore
\[
\partial_i^p
=
(B^{-1}\delta_i)^p
=
B^{-p}\delta_i^p.
\]
Using Proposition~\ref{prop:delta-p-relation}, we obtain
\[
\partial_i^p
=
B^{-p}B^{p-1}\delta_i
=
B^{-1}\delta_i
=
\partial_i.
\]

Consider \(S_i\) as an \(\F_p\)-vector space and set
\[
h(T):=T^p-T.
\]
The identity \(\partial_i^p=\partial_i\) says that
\[
h(\partial_i)=0.
\]
Over \(\F_p\),
\[
h(T)=\prod_{\lambda\in\F_p}(T-\lambda).
\]
Its roots are distinct, because \(h'(T)=pT^{p-1}-1=-1\). For each \(\lambda\in\F_p\), define
\[
\ell_\lambda(T)
:=
\prod_{\substack{\mu\in\F_p\\ \mu\ne\lambda}}
\frac{T-\mu}{\lambda-\mu}.
\]
Every denominator \(\lambda-\mu\) is a nonzero element of \(\F_p\), so \(\ell_\lambda(T)\in\F_p[T]\). The Lagrange interpolation identities give, modulo \(h(T)\),
\[
\sum_{\lambda\in\F_p}\ell_\lambda(T)=1,
\]
\[
T\ell_\lambda(T)=\lambda\ell_\lambda(T),
\]
and
\[
\ell_\lambda(T)\ell_\mu(T)
=
\begin{cases}
\ell_\lambda(T),&\lambda=\mu,\\
0,&\lambda\ne\mu.
\end{cases}
\]
Applying these polynomial identities to \(\partial_i\) is legitimate because \(h(\partial_i)=0\). Thus the endomorphisms
\[
\pi_\lambda:=\ell_\lambda(\partial_i)
\]
satisfy
\[
\sum_{\lambda\in\F_p}\pi_\lambda=\operatorname{id}_{S_i},
\qquad
\partial_i\pi_\lambda=\lambda\pi_\lambda,
\]
and
\[
\pi_\lambda\pi_\mu
=
\begin{cases}
\pi_\lambda,&\lambda=\mu,\\
0,&\lambda\ne\mu.
\end{cases}
\]
For every \(f\in S_i\), we therefore have
\[
f=\sum_{\lambda\in\F_p}\pi_\lambda(f),
\qquad
\pi_\lambda(f)\in S_i^{(\lambda)}.
\]
The orthogonality of the projectors makes this sum direct, proving
\[
S_i=\bigoplus_{\lambda\in\F_p}S_i^{(\lambda)}
\]
and the projector formula \eqref{eq:weight-projector}.

Let \(f\in S_i^{(\lambda)}\) and \(g\in S_i^{(\mu)}\). Since \(\partial_i\) is a derivation,
\[
\partial_i(fg)
=
\partial_i(f)g+f\partial_i(g)
=
\lambda fg+\mu fg
=
(\lambda+\mu)fg.
\]
Hence
\[
fg\in S_i^{(\lambda+\mu)},
\]
which proves \eqref{eq:weight-product}. Furthermore,
\[
\partial_i(Q_{n,0})
=
B^{-1}\delta_i(Q_{n,0})
=
B^{-1}BQ_{n,0}
=
Q_{n,0},
\]
so \(Q_{n,0}\in S_i^{(1)}\). Also,
\[
\partial_i(B)=0
\]
and Proposition~\ref{prop:invariant-ratios} gives
\[
\partial_i(I_s)=B^{-1}\delta_i(I_s)=0.
\]
Thus \(B\) and every \(I_s\) have weight \(0\).

We next prove \eqref{eq:weight-iterate}. Put
\[
D:=\StDeltai=(-1)^nBQ_{n,0}\partial_i.
\]
For \(f\in S_i^{(\lambda)}\),
\[
D(f)=(-1)^n\lambda BQ_{n,0}f.
\]
Because \(B\) has weight \(0\) and \(Q_{n,0}\) has weight \(1\), the element
\[
B^mQ_{n,0}^{\,m}f
\]
has weight \(\lambda+m\). We argue by induction on \(m\). The case \(m=1\) is the preceding formula. Suppose
\[
D^m(f)
=
(-1)^{mn}B^mQ_{n,0}^{\,m}
\left(\prod_{r=0}^{m-1}(\lambda+r)\right)f.
\]
Applying \(D\) and using the weight \(\lambda+m\) of \(B^mQ_{n,0}^{\,m}f\), we obtain
\begin{align*}
D^{m+1}(f)
&=
(-1)^{mn}
\left(\prod_{r=0}^{m-1}(\lambda+r)\right)
D\big(B^mQ_{n,0}^{\,m}f\big)\\
&=
(-1)^{mn}
\left(\prod_{r=0}^{m-1}(\lambda+r)\right)
(-1)^n(\lambda+m)B^{m+1}Q_{n,0}^{\,m+1}f\\
&=
(-1)^{(m+1)n}B^{m+1}Q_{n,0}^{\,m+1}
\left(\prod_{r=0}^{m}(\lambda+r)\right)f.
\end{align*}
This proves \eqref{eq:weight-iterate}. In particular, when \(m=p\), the factors
\[
\lambda,\lambda+1,\dots,\lambda+p-1
\]
run through all elements of \(\F_p\), so their product is zero.

It remains to determine the kernels and images. On \(S_i^{(\lambda)}\),
\[
\delta_i=B\partial_i
\]
acts as multiplication by \(\lambda B\). Since \(B\) is a unit in \(S_i\), this action is zero precisely for \(\lambda=0\), and it is an automorphism of \(S_i^{(\lambda)}\) for every \(\lambda\ne0\). Therefore
\[
\ker(\delta_i)=S_i^{(0)},
\qquad
\im(\delta_i)=\bigoplus_{\lambda\in\F_p^\times}S_i^{(\lambda)}.
\]

Similarly,
\[
\StDeltai=(-1)^nBQ_{n,0}\partial_i.
\]
Both \(B\) and \(Q_{n,0}\) are units in \(S_i\). The operation \(\StDeltai\) vanishes on \(S_i^{(0)}\). For \(\lambda\ne0\), it maps \(S_i^{(\lambda)}\) to \(S_i^{(\lambda+1)}\) by multiplication by the unit
\[
(-1)^n\lambda BQ_{n,0}.
\]
This map is bijective. Indeed, the quotient rule gives
\[
\partial_i\big((BQ_{n,0})^{-1}\big)
=
-(BQ_{n,0})^{-2}\partial_i(BQ_{n,0})
=
-(BQ_{n,0})^{-1},
\]
because \(B\) has weight \(0\) and \(Q_{n,0}\) has weight \(1\). Thus \((BQ_{n,0})^{-1}\) has weight \(-1\). If \(g\in S_i^{(\lambda+1)}\), then
\[
f:=\big((-1)^n\lambda BQ_{n,0}\big)^{-1}g
\]
has weight \(-1+(\lambda+1)=\lambda\), and direct substitution gives \(\StDeltai(f)=g\). Hence the image consists exactly of the weights \(\lambda+1\) with \(\lambda\ne0\), namely all weights except \(1\). Thus
\[
\ker(\StDeltai)=S_i^{(0)},
\qquad
\im(\StDeltai)
=
\bigoplus_{\lambda\in\F_p\setminus\{1\}}S_i^{(\lambda)}.
\]
\qed

\section{Recovery and strengthening of previously known formulas}\label{sec:explicit}

The purpose of this section is to show that the normalized-derivation formalism gives a unified conceptual framework and recovers several previously known first-order formulas and upgrades them to closed formulas for all higher iterates. We do this in two ranges where explicit formulas are available in the literature: first, the classical range $1\le i\le n$, governed by the classical formula (see \cite{SmithSwitzer,WilkersonPrimer}; recorded by Sum in \cite[Proposition~2.3]{Sum}); and second, the low extra range $i=n+1,n+2$, governed by Sum's formulas recorded in \cite[Corollary~2.6]{Sum}; see also \cite[Remark~3.10]{Sum2007}.

Throughout this section we adopt the convention
\[
Q_{n,t}:=0 \qquad \text{for } t<0.
\]

\subsection{The classical range \texorpdfstring{$1\le i\le n$}{1<=i<=n}}

\begin{proposition}[Classical range]\label{prop:classical-range}
Let \(1\le i<n\). Then
\[
\delta_i(Q_{n,s})
=
\begin{cases}
(-1)^{n+i-1}, & s=i,\\
0, & s\neq i,
\end{cases}
\]
and, for every integer \(m\ge1\),
\[
\delta_i^{\,m}(Q_{n,s})
=
\begin{cases}
(-1)^{n+i-1}, & m=1 \text{ and } s=i,\\
0, & \text{otherwise}.
\end{cases}
\]
Consequently,
\[
(\mathrm{St}^{\Delta_i})^{\,m}(Q_{n,s})
=
\begin{cases}
(-1)^{i-1}Q_{n,0}, & m=1 \text{ and } s=i,\\
0, & \text{otherwise}.
\end{cases}
\]

For \(i=n\), one has
\[
\delta_n(Q_{n,s})=Q_{n,s}
\qquad (0\le s<n),
\]
hence
\[
\delta_n^{\,m}(Q_{n,s})=Q_{n,s}
\qquad (m\ge1),
\]
and
\[
(\mathrm{St}^{\Delta_n})^{\,m}(Q_{n,s})
=
(-1)^{mn}m!\,Q_{n,0}^{\,m}Q_{n,s}
\qquad (m\ge1).
\]
In particular,
\[
(\mathrm{St}^{\Delta_n})^{\,m}(Q_{n,s})=0
\qquad (m\ge p).
\]
For \(m=1\), these identities recover the classical formula of Smith--Switzer and Wilkerson; see \cite{SmithSwitzer,WilkersonPrimer} and \cite[Proposition~2.3]{Sum}.
\end{proposition}

\begin{proof}
By Theorem~\ref{thm:CF}, for every $i\ge1$ and every $s$ one has
\[
\mathrm{St}^{\Delta_i}(Q_{n,s})
=
(-1)^nQ_{n,0}\big(A_s+BQ_{n,s}\big),
\]
where
\[
A_s=P_{n,i,s}^{\,p},
\qquad
B=R_{n,i}^{\,p}.
\]

\smallskip
\noindent
\textit{First suppose that \(1\le i<n\).}
In this range, by \cite[Proposition~2.3]{Sum}, one has
\[
\mathrm{St}^{\Delta_i}(Q_{n,s})
=
\begin{cases}
(-1)^{i-1}Q_{n,0}, & s=i,\\
0, & s\neq i.
\end{cases}
\]
Taking $s=0$ in the general expression
\[
\mathrm{St}^{\Delta_i}(Q_{n,s})
=
(-1)^nQ_{n,0}\big(A_s+BQ_{n,s}\big)
\]
gives
\[
0
=
\mathrm{St}^{\Delta_i}(Q_{n,0})
=
(-1)^nQ_{n,0}\big(A_0+BQ_{n,0}\big).
\]
Since $A_0=P_{n,i,0}^{\,p}=0$, this becomes
\[
(-1)^nBQ_{n,0}^{\,2}=0.
\]
The Dickson algebra $D_n$ is a domain and $Q_{n,0}\ne0$, hence
\[
B=0.
\]
Therefore the general formula reduces to
\[
\mathrm{St}^{\Delta_i}(Q_{n,s})
=
(-1)^nQ_{n,0}A_s.
\]
Comparing with this classical formula and again using that multiplication by $Q_{n,0}$ is injective in $D_n$, we obtain
\[
A_s=
\begin{cases}
(-1)^{n+i-1}, & s=i,\\
0, & s\neq i.
\end{cases}
\]
Hence
\[
\delta_i(Q_{n,s})=A_s
=
\begin{cases}
(-1)^{n+i-1}, & s=i,\\
0, & s\neq i.
\end{cases}
\]
Since each $A_s$ is a scalar in $\F_p$, applying $\delta_i$ once more gives
\[
\delta_i^{\,m}(Q_{n,s})=0
\qquad \text{for all } m\ge 2.
\]
This proves the asserted formula for $\delta_i^{\,m}(Q_{n,s})$.

Moreover, since $B=0$, Proposition~\ref{prop:normalized-derivation} gives
\[
\delta_i(Q_{n,0})=B\,Q_{n,0}=0,
\]
equivalently,
\[
\mathrm{St}^{\Delta_i}(Q_{n,0})=0.
\]
Therefore
\[
\mathrm{St}^{\Delta_i}(Q_{n,s})
=
(-1)^nQ_{n,0}\delta_i(Q_{n,s})
=
\begin{cases}
(-1)^{i-1}Q_{n,0}, & s=i,\\
0, & s\neq i,
\end{cases}
\]
which is exactly the classical formula in this range. Applying $\mathrm{St}^{\Delta_i}$ once more, we obtain
\[
(\mathrm{St}^{\Delta_i})^2(Q_{n,s})=0
\qquad \text{for all } s,
\]
because after the first application the result is either $0$ or a scalar multiple of $Q_{n,0}$, and $\mathrm{St}^{\Delta_i}(Q_{n,0})=0$. Hence
\[
(\mathrm{St}^{\Delta_i})^{\,m}(Q_{n,s})=0
\qquad \text{for all } m\ge 2.
\]
This proves the assertions for \(1\le i<n\), and the case \(m=1\) recovers the classical formula in \cite[Proposition~2.3]{Sum}.

\smallskip
\noindent
\textit{Now suppose that \(i=n\).}
By \cite[Proposition~2.3]{Sum}, we have
\[
\mathrm{St}^{\Delta_n}(Q_{n,s})=(-1)^nQ_{n,0}Q_{n,s}.
\]
Taking $s=0$ in the general first-order formula gives
\[
(-1)^nQ_{n,0}^{\,2}
=
\mathrm{St}^{\Delta_n}(Q_{n,0})
=
(-1)^nQ_{n,0}\big(A_0+BQ_{n,0}\big).
\]
Since $A_0=P_{n,n,0}^{\,p}=0$, and since $D_n$ is a domain with $Q_{n,0}\ne0$, it follows that
\[
B=1.
\]
For general $s$, comparison with the classical formula gives
\[
(-1)^nQ_{n,0}(A_s+Q_{n,s})
=
(-1)^nQ_{n,0}Q_{n,s}.
\]
Again using injectivity of multiplication by $Q_{n,0}$, we get
\[
A_s=0
\qquad (0\le s<n).
\]
Therefore
\[
\delta_n(Q_{n,s})=A_s+BQ_{n,s}=Q_{n,s},
\]
and inductively
\[
\delta_n^{\,m}(Q_{n,s})=Q_{n,s}
\qquad (m\ge 1).
\]
Substituting $A_s=0$ and $B=1$ into Theorem~\ref{thm:iterate-St}, we get
\[
(\mathrm{St}^{\Delta_n})^{\,m}(Q_{n,s})
=
(-1)^{mn}m!\,Q_{n,0}^{\,m}Q_{n,s}
\qquad \text{for all } m\ge 1.
\]
For $m=1$, this is precisely the classical formula, while for $m\ge 2$ it gives the higher iterates. The vanishing for $m\ge p$ follows from the factor $m!$ in $\F_p$.
\end{proof}

\begin{remark}
For \(1\le i<n\), the only nonzero higher iterate is the first one on the single generator $Q_{n,i}$. The case \(i=n\) is qualitatively different: the normalized derivation acts by the identity on the Dickson generators, so the full higher-iterate pattern is determined entirely by the factorial term appearing in Theorem~\ref{thm:iterate-St}.
\end{remark}

The explicit formulas of Proposition~\ref{prop:classical-range} also lead to exact kernel and image statements.

\begin{theorem}[Exact kernel and image for $1\le i<n$]\label{thm:classical-kernel-image}
Let $1\le i<n$ and set
\[
C_i:=\F_p[\Qnzero,Q_{n,1},\dots,\widehat{Q_{n,i}},\dots,Q_{n,n-1}].
\]
Then
\[
D_n=C_i[Q_{n,i}]=\bigoplus_{r=0}^{p-1} C_i[Q_{n,i}^p]Q_{n,i}^r
\]
as a $C_i[Q_{n,i}^p]$-module, and the normalized derivation satisfies
\[
\ker(\delta_i)=C_i[Q_{n,i}^p],
\qquad
\im(\delta_i)=\bigoplus_{r=0}^{p-2} C_i[Q_{n,i}^p]Q_{n,i}^r.
\]
Consequently,
\[
\ker(\StDeltai)=C_i[Q_{n,i}^p],
\qquad
\im(\StDeltai)=\Qnzero\bigoplus_{r=0}^{p-2} C_i[Q_{n,i}^p]Q_{n,i}^r.
\]
\end{theorem}

\begin{proof}
By Proposition~\ref{prop:classical-range}, one has $\delta_i(D_n)\subseteq D_n.$
Thus, in the sense of Remark~\ref{rem:delta-restriction}, we may view $\delta_i: D_n\to D_n$ as the restricted derivation. Since $D_n=C_i[Q_{n,i}]$ is a polynomial algebra in the single variable $Q_{n,i}$ over $C_i$, every element of $D_n$ can be written uniquely in the form
\[
f=\sum_{r=0}^{p-1} Q_{n,i}^r f_r(Q_{n,i}^p)
\qquad \text{with } f_r(Q_{n,i}^p)\in C_i[Q_{n,i}^p].
\]
This gives the stated direct-sum decomposition.

Let
\[
c:=(-1)^{n+i-1}\in \F_p^{\times}.
\]
By Proposition~\ref{prop:classical-range}, one has
\[
\delta_i(Q_{n,i})=c,
\qquad
\delta_i(Q_{n,s})=0 \quad (s\ne i).
\]
Hence $\delta_i$ annihilates $C_i$ and also annihilates $Q_{n,i}^p$ because $\delta_i$ kills $p$-th powers. Applying the Leibniz rule to the above decomposition of $f$, we obtain
\[
\delta_i(f)=c\sum_{r=1}^{p-1} rQ_{n,i}^{r-1}f_r(Q_{n,i}^p).
\]
Since $c\ne 0$ and each $r\in\{1,\dots,p-1\}$ is nonzero in $\F_p$, it follows that $\delta_i(f)=0$ if and only if
\[
f_r(Q_{n,i}^p)=0
\qquad \text{for all } r\ge 1.
\]
Therefore
\[
\ker(\delta_i)=C_i[Q_{n,i}^p].
\]

Now let
\[
g=\sum_{r=0}^{p-2} Q_{n,i}^r g_r(Q_{n,i}^p)
\qquad \text{with } g_r(Q_{n,i}^p)\in C_i[Q_{n,i}^p].
\]
Define
\[
h:=c^{-1}\sum_{r=0}^{p-2}(r+1)^{-1}Q_{n,i}^{r+1}g_r(Q_{n,i}^p).
\]
Then the same computation shows that
\[
\delta_i(h)=g.
\]
Thus every such $g$ lies in $\im(\delta_i)$, and conversely the displayed formula for $\delta_i(f)$ shows that no term of $Q_{n,i}$-degree $p-1$ can occur in an image. Hence
\[
\im(\delta_i)=\bigoplus_{r=0}^{p-2} C_i[Q_{n,i}^p]Q_{n,i}^r.
\]

Finally, in the range $1\le i<n$ one has $\delta_i(\Qnzero)=0$ by Proposition~\ref{prop:classical-range}, so
\[
\StDeltai=(-1)^n\Qnzero\,\delta_i
\]
on $D_n$. Since $D_n$ is a domain and $\Qnzero\ne 0$, multiplication by $\Qnzero$ is injective. Therefore
\[
\ker(\StDeltai)=\ker(\delta_i)=C_i[Q_{n,i}^p].
\]
Moreover,
\[
\im(\StDeltai)=\Qnzero\im(\delta_i)
=\Qnzero\bigoplus_{r=0}^{p-2} C_i[Q_{n,i}^p]Q_{n,i}^r.
\]
This proves the theorem.
\end{proof}

\begin{theorem}[Generator-degree description for $i=n$]\label{thm:generator-degree-n}
Define an auxiliary grading on $D_n$ by declaring each Dickson generator $Q_{n,s}$ to have degree $1$. For $d\ge 0$, let $D_n^{\langle d\rangle}$ denote the homogeneous piece of auxiliary degree $d$. Then
\[
D_n=\bigoplus_{d\ge 0} D_n^{\langle d\rangle},
\]
and, for every $F\in D_n^{\langle d\rangle}$,
\[
\delta_n(F)=dF.
\]
Consequently,
\[
\ker(\delta_n)=\bigoplus_{\substack{d\ge 0 \\ p\mid d}} D_n^{\langle d\rangle},
\qquad
\im(\delta_n)=\bigoplus_{\substack{d\ge 0 \\ p\nmid d}} D_n^{\langle d\rangle}.
\]
Moreover,
\[
\ker(\mathrm{St}^{\Delta_n})=\bigoplus_{\substack{d\ge 0 \\ p\mid d}} D_n^{\langle d\rangle},
\qquad
\im(\mathrm{St}^{\Delta_n})=\Qnzero\bigoplus_{\substack{d\ge 0 \\ p\nmid d}} D_n^{\langle d\rangle}.
\]
\end{theorem}

\begin{proof}
By Proposition~\ref{prop:classical-range}, one has
\[
\delta_n(Q_{n,s})=Q_{n,s}
\qquad (0\le s<n).
\]
Let
\[
M=Q_{n,0}^{e_0}Q_{n,1}^{e_1}\cdots Q_{n,n-1}^{e_{n-1}}
\]
be a monomial of auxiliary degree
\[
d=e_0+e_1+\cdots+e_{n-1}.
\]
Repeated use of the Leibniz rule gives
\[
\delta_n(M)=(e_0+e_1+\cdots+e_{n-1})M=dM.
\]
By linearity, the same identity holds for every $F\in D_n^{\langle d\rangle}$.

Thus $\delta_n$ acts on each graded piece $D_n^{\langle d\rangle}$ as multiplication by the scalar $d\in\F_p$. If $p\mid d$, then this scalar is $0$, so $D_n^{\langle d\rangle}\subseteq \ker(\delta_n)$. If $p\nmid d$, then the scalar $d$ is invertible in $\F_p$, so $\delta_n$ restricts to an automorphism of $D_n^{\langle d\rangle}$. Therefore
\[
\ker(\delta_n)=\bigoplus_{\substack{d\ge 0 \\ p\mid d}} D_n^{\langle d\rangle},
\qquad
\im(\delta_n)=\bigoplus_{\substack{d\ge 0 \\ p\nmid d}} D_n^{\langle d\rangle}.
\]

Finally,
\[
\mathrm{St}^{\Delta_n}=(-1)^n\Qnzero\,\delta_n
\]
by Proposition~\ref{prop:classical-range}. Since $D_n$ is a domain and $\Qnzero\ne 0$, multiplication by $\Qnzero$ is injective, hence
\[
\ker(\mathrm{St}^{\Delta_n})=\ker(\delta_n)=\bigoplus_{\substack{d\ge 0 \\ p\mid d}} D_n^{\langle d\rangle}.
\]
Moreover,
\[
\im(\mathrm{St}^{\Delta_n})=\Qnzero\im(\delta_n)
=\Qnzero\bigoplus_{\substack{d\ge 0 \\ p\nmid d}} D_n^{\langle d\rangle}.
\]
This proves the theorem.
\end{proof}

\begin{remark}
The grading used in Theorem~\ref{thm:generator-degree-n} is the polynomial grading in the Dickson generators, not the natural internal grading inherited from $P_n$, since the Dickson generators have different internal degrees.
\end{remark}

\subsection{The low extra range \texorpdfstring{$i=n+1,n+2$}{i=n+1,n+2}}

\begin{proposition}[Low extra range]\label{prop:low-extra-range}
The formulas below recover Sum's first-order identities \cite[Corollary~2.6]{Sum}; see also \cite[Remark~3.10]{Sum2007}, and extend them to every higher iterate.

For \(i=n+1\) and \(0\le s<n\),
\[
\delta_{n+1}(Q_{n,s})
=
Q_{n,n-1}^{\,p}Q_{n,s}-Q_{n,s-1}^{\,p}.
\]
For every integer \(m\ge1\),
\begin{equation}\label{eq:nplus1-delta}
\delta_{n+1}^{\,m}(Q_{n,s})
=
Q_{n,n-1}^{\,mp}Q_{n,s}
-
Q_{n,n-1}^{\,(m-1)p}Q_{n,s-1}^{\,p},
\end{equation}
and
\begin{equation}\label{eq:nplus1-St}
(\mathrm{St}^{\Delta_{n+1}})^{\,m}(Q_{n,s})
=
(-1)^{mn}m!\,Q_{n,0}^{\,m}
\left(
Q_{n,n-1}^{\,mp}Q_{n,s}
-
Q_{n,n-1}^{\,(m-1)p}Q_{n,s-1}^{\,p}
\right).
\end{equation}

For \(i=n+2\) and \(0\le s<n\), set
\[
A^{(n+2)}_s
:=
Q_{n,s-2}^{\,p^2}
-
Q_{n,s-1}^{\,p}Q_{n,n-1}^{\,p^2},
\qquad
B^{(n+2)}
:=
Q_{n,n-1}^{\,p^2+p}
-
Q_{n,n-2}^{\,p^2}.
\]
Equivalently,
\[
A^{(n+2)}_s
=
\big(Q_{n,s-2}^{\,p}-Q_{n,s-1}Q_{n,n-1}^{\,p}\big)^p,
\]
and
\[
B^{(n+2)}
=
\big(Q_{n,n-1}^{\,p+1}-Q_{n,n-2}^{\,p}\big)^p.
\]
Then
\[
\delta_{n+2}(Q_{n,s})
=
A^{(n+2)}_s+B^{(n+2)}Q_{n,s},
\]
and, for every integer \(m\ge1\),
\begin{equation}\label{eq:nplus2-delta}
\delta_{n+2}^{\,m}(Q_{n,s})
=
\big(B^{(n+2)}\big)^mQ_{n,s}
+
\big(B^{(n+2)}\big)^{m-1}A^{(n+2)}_s,
\end{equation}
while
\begin{equation}\label{eq:nplus2-St}
(\mathrm{St}^{\Delta_{n+2}})^{\,m}(Q_{n,s})
=
(-1)^{mn}m!\,Q_{n,0}^{\,m}
\left(
\big(B^{(n+2)}\big)^mQ_{n,s}
+
\big(B^{(n+2)}\big)^{m-1}A^{(n+2)}_s
\right).
\end{equation}

For both \(i=n+1\) and \(i=n+2\),
\[
(\mathrm{St}^{\Delta_i})^{\,m}(Q_{n,s})=0
\qquad (m\ge p).
\]
The specialization \(m=1\) is exactly \cite[Corollary~2.6]{Sum}.
\end{proposition}

\begin{proof}
By Theorem~\ref{thm:CF}, for every $i\ge1$ and every $s$ one has
\[
\mathrm{St}^{\Delta_i}(Q_{n,s})
=
(-1)^nQ_{n,0}\big(A_s+BQ_{n,s}\big),
\]
where
\[
A_s=P_{n,i,s}^{\,p},
\qquad
B=R_{n,i}^{\,p}.
\]

\smallskip
\noindent
\textit{First suppose that \(i=n+1\).}
In this case, Corollary~2.6 of Sum \cite{Sum} gives the first-order formula
\[
\mathrm{St}^{\Delta_{n+1}}(Q_{n,s})
=
(-1)^nQ_{n,0}\big(-Q_{n,s-1}^{\,p}+Q_{n,n-1}^{\,p}Q_{n,s}\big).
\]
Thus, in the general first-order expression
\[
\mathrm{St}^{\Delta_{n+1}}(Q_{n,s})
=
(-1)^nQ_{n,0}\big(A_s+BQ_{n,s}\big),
\]
we read off
\[
A_s=-Q_{n,s-1}^{\,p},
\qquad
B=Q_{n,n-1}^{\,p}.
\]
Hence
\[
\delta_{n+1}(Q_{n,s})
=
A_s+BQ_{n,s}
=
Q_{n,n-1}^{\,p}Q_{n,s}-Q_{n,s-1}^{\,p},
\]
which is exactly the $m=1$ formula above, and therefore recovers Sum's formula after multiplying back by $(-1)^nQ_{n,0}$.

Now Theorem~\ref{thm:iterate-St} applies with these values of $A_s$ and $B$, and gives
\[
\delta_{n+1}^{\,m}(Q_{n,s})
=
B^mQ_{n,s}+B^{m-1}A_s
=
Q_{n,n-1}^{\,mp}Q_{n,s}
-
Q_{n,n-1}^{\,(m-1)p}Q_{n,s-1}^{\,p},
\]
which is \eqref{eq:nplus1-delta}. Substituting the same values into the second formula of Theorem~\ref{thm:iterate-St}, we obtain
\[
(\mathrm{St}^{\Delta_{n+1}})^{\,m}(Q_{n,s})
=
(-1)^{mn}m!\,Q_{n,0}^{\,m}
\left(
Q_{n,n-1}^{\,mp}Q_{n,s}
-
Q_{n,n-1}^{\,(m-1)p}Q_{n,s-1}^{\,p}
\right),
\]
which is \eqref{eq:nplus1-St}.

\smallskip
\noindent
\textit{Now suppose that \(i=n+2\).}
Again, Corollary~2.6 of Sum \cite{Sum} gives
\[
\mathrm{St}^{\Delta_{n+2}}(Q_{n,s})
=
(-1)^nQ_{n,0}
\Big(
Q_{n,s-2}^{\,p^2}
-
Q_{n,s-1}^{\,p}Q_{n,n-1}^{\,p^2}
+
\big(Q_{n,n-1}^{\,p^2+p}-Q_{n,n-2}^{\,p^2}\big)Q_{n,s}
\Big).
\]
Thus, in the general first-order expression
\[
\mathrm{St}^{\Delta_{n+2}}(Q_{n,s})
=
(-1)^nQ_{n,0}\big(A_s+BQ_{n,s}\big),
\]
we read off
\[
A_s=A^{(n+2)}_s
=
Q_{n,s-2}^{\,p^2}
-
Q_{n,s-1}^{\,p}Q_{n,n-1}^{\,p^2},
\]
and
\[
B=B^{(n+2)}
=
Q_{n,n-1}^{\,p^2+p}
-
Q_{n,n-2}^{\,p^2}.
\]
Therefore
\[
\delta_{n+2}(Q_{n,s})
=
A^{(n+2)}_s+B^{(n+2)}Q_{n,s},
\]
which is precisely the $m=1$ formula above, and hence recovers Sum's formula after multiplying by $(-1)^nQ_{n,0}$.

Applying Theorem~\ref{thm:iterate-St} with these values of $A_s$ and $B$, we get
\[
\delta_{n+2}^{\,m}(Q_{n,s})
=
\big(B^{(n+2)}\big)^mQ_{n,s}
+
\big(B^{(n+2)}\big)^{m-1}A^{(n+2)}_s,
\]
which is \eqref{eq:nplus2-delta}, and
\[
(\mathrm{St}^{\Delta_{n+2}})^{\,m}(Q_{n,s})
=
(-1)^{mn}m!\,Q_{n,0}^{\,m}
\left(
\big(B^{(n+2)}\big)^mQ_{n,s}
+
\big(B^{(n+2)}\big)^{m-1}A^{(n+2)}_s
\right),
\]
which is \eqref{eq:nplus2-St}.

Finally, in both cases the vanishing for $m\ge p$ follows immediately from the factor $m!$ in the formulas for
\[
(\mathrm{St}^{\Delta_i})^{\,m}(Q_{n,s}).
\]
Thus the case $m=1$ recovers \cite[Corollary~2.6]{Sum}, while the formulas for $m\ge2$ provide the higher iterates.
\end{proof}

\section{Koszul-type complexes and structural comparisons with Margolis homology}\label{sec:koszul}

The purpose of this section is twofold. First, we recall the structural form of Tuan's computation \cite{Tuan}, which is genuinely a Margolis-homology calculation using the usual Milnor primitives $Q_j$. Second, we explain a more modest Koszul construction naturally attached to the normalized-ratio coefficients
\[
\StDeltai(R_s)=(-1)^nA_s.
\]
This construction is useful as an organizing device, but it should not be read as a Margolis-homology computation for $D_n$.

\begin{remark}[A basic distinction]\label{rem:not-margolis}
The usual Margolis homology considered in the cited works is built from Milnor primitives $Q_j$, which belong to the exterior $\tau$-side of the Steenrod--Milnor notation and act as square-zero differentials in the relevant settings. The operation studied in this paper is
\[
\StDeltai=\mathrm{St}^{\emptyset,\Delta_i},
\]
which belongs instead to the polynomial $\xi$-side. Although Corollary~\ref{cor:iterate-ideal} shows that
\[
(\StDeltai)^p=0
\qquad\text{on }D_n,
\]
the operation $\StDeltai$ is not generally square-zero.

For instance, in the classical case $i=n$, Proposition~\ref{prop:classical-range} gives
\[
(\mathrm{St}^{\Delta_n})^2(Q_{n,s})
=
2Q_{n,0}^{\,2}Q_{n,s},
\]
which is nonzero in $D_n$ for odd $p$. Hence $\StDeltai$ is not the differential used in Tuan's Margolis-homology calculation \cite{Tuan}. The discussion below is therefore an ordinary commutative-algebra Koszul construction inspired by the normalized-ratio action, not a direct Margolis-homology computation.
\end{remark}

\subsection{The Koszul structure of prior Margolis homology computations}

We first recall Tuan's theorem \cite{Tuan}, solely for structural comparison.

\begin{theorem}[{\cite[Theorem~1.3]{Tuan}}]\label{thm:Tuan-original}
For every maximal elementary abelian \(p\)-subgroup \(A\) of \(M_n\), the Margolis homology of \(\mathrm{ImRes}(A,M_n)\) with respect to \(Q_j\) is given by
\[
H_{*}(\mathrm{ImRes}(A,M_n);Q_j)
\cong
\frac{P(y_1,\dots,y_n,V_{n+1})}{(y_1^{p^j},\dots,y_n^{p^j})}.
\]
\end{theorem}

\begin{note}
Here $M_n$ denotes the extra-special $p$-group of order $p^{2n+1}$, $A$ is a maximal elementary abelian $p$-subgroup of $M_n$, and $\mathrm{ImRes}(A,M_n)$ denotes the image of the restriction map
\[
H^*(M_n)\longrightarrow H^*(A).
\]
In Tuan's setting \cite{Tuan}, one has a differential graded algebra
\[
\mathcal{R}_n=E(x_1,\dots,x_n)\otimes P(y_1,\dots,y_n,V_{n+1})
\]
with differential given by the Milnor primitive $Q_j$:
\[
Q_j(x_s)=y_s^{p^j},
\qquad
Q_j(y_s)=0,
\qquad
Q_j(V_{n+1})=0
\qquad\text{(see \cite[Lemma~2.1]{Tuan})}.
\]
Thus the differential is exactly the Koszul differential associated to the regular sequence
\[
(y_1^{p^j},\dots,y_n^{p^j})
\]
in the polynomial ring $P(y_1,\dots,y_n,V_{n+1})$. This is the mechanism behind Theorem~\ref{thm:Tuan-original}.
\end{note}

We now isolate the standard Koszul-complex interpretation of Theorem~\ref{thm:Tuan-original}.

\begin{proposition}[Koszul rephrasing of Theorem~\ref{thm:Tuan-original}]\label{prop:Tuan-koszul}
Let
\[
S_T=P(y_1,\dots,y_n,V_{n+1})
\]
and let
\[
\boldsymbol{a}^{(j)}=(y_1^{p^j},\dots,y_n^{p^j})\in S_T^n.
\]
Since $\boldsymbol{a}^{(j)}$ is an $S_T$-regular sequence, the Koszul complex
\[
K\big(\boldsymbol{a}^{(j)};S_T\big)
\cong
\big(E(x_1,\dots,x_n)\otimes S_T,Q_j\big)
\]
has homology
\[
H_0\cong S_T/(y_1^{p^j},\dots,y_n^{p^j}),
\qquad
H_{>0}=0.
\]
\end{proposition}

\begin{proof}
The ring $S_T=P(y_1,\dots,y_n,V_{n+1})$ is a polynomial algebra over $\F_p$. Hence the variables $y_1,\dots,y_n$ form an $S_T$-regular sequence. Since powers of a regular sequence remain regular, it follows that
\[
\boldsymbol a^{(j)}=(y_1^{p^j},\dots,y_n^{p^j})
\]
is also an $S_T$-regular sequence.

By definition, the Koszul complex $K(\boldsymbol a^{(j)};S_T)$ is the differential graded algebra
\[
\big(E(x_1,\dots,x_n)\otimes S_T,d\big),
\]
where $d|_{S_T}=0$ and
\[
d(x_s)=y_s^{p^j}
\qquad (1\le s\le n).
\]
This is exactly the differential graded algebra $(E(x_1,\dots,x_n)\otimes S_T,Q_j)$ appearing in Tuan's setting. Since $\boldsymbol a^{(j)}$ is an $S_T$-regular sequence, the standard exactness theorem for Koszul complexes on regular sequences (see Eisenbud \cite[Corollary~17.5]{Eisenbud}) gives
\[
H_{>0}\big(K(\boldsymbol a^{(j)};S_T)\big)=0,
\qquad
H_0\big(K(\boldsymbol a^{(j)};S_T)\big)\cong S_T/(y_1^{p^j},\dots,y_n^{p^j}).
\]
This proves the proposition.
\end{proof}

\subsection{Koszul complexes from normalized Dickson coefficients}

We now return to the Dickson algebra. Let
\[
S:=D_n[\Qnzero^{-1}]
\cong
\F_p[\Qnzero^{\pm1},R_1,\dots,R_{n-1}].
\]
By Theorem~\ref{thm:ratio-action}, for each $1\le s\le n-1$ we have a coefficient
\[
c_s:=\StDeltai(R_s)=(-1)^nA_s\in S.
\]
This sequence gives an ordinary Koszul complex over $S$.

\begin{definition}[The normalized-ratio Koszul complex]\label{def:ratio-koszul}
Let $E(e_1,\dots,e_{n-1})$ denote the exterior $\F_p$-algebra on generators
$e_1,\dots,e_{n-1}$, each of exterior degree $1$. The \emph{normalized-ratio Koszul complex} is the standard Koszul complex on the sequence $(c_1,\dots,c_{n-1})\subset S$:
\[
K_i(S):=\big(E(e_1,\dots,e_{n-1})\otimes S,d_i\big),
\]
where $d_i$ is the $S$-linear differential of exterior degree $-1$ determined by
\[
d_i|_S=0,
\qquad
d_i(e_s)=c_s=\StDeltai(R_s)
\quad (1\le s\le n-1),
\]
and extended to all of $E(e_1,\dots,e_{n-1})\otimes S$ by the graded Leibniz rule.
\end{definition}

\begin{remark}
When \(n=1\), the coefficient sequence and the exterior generating set are empty. With the standard convention for the Koszul complex on the empty sequence,
\[
K_i(S)=S
\]
concentrated in exterior degree \(0\), so \(H_0(K_i(S))=S\) and \(H_{>0}(K_i(S))=0\).

The complex in Definition~\ref{def:ratio-koszul} is a standard commutative-algebra Koszul complex. Its differential is not the Steenrod operation $\StDeltai$; rather, its coefficients are the values of $\StDeltai$ on the normalized ratios $R_s$. This distinction is important because $\StDeltai$ is not the square-zero Milnor primitive used in Margolis homology.
\end{remark}

\begin{example}[The localized complex can become trivial]\label{ex:koszul-i-n+1}
Assume \(n\ge2\) and let \(i=n+1\). Proposition~\ref{prop:low-extra-range} gives
\[
c_s=\StDeltai(R_s)=(-1)^n(-Q_{n,s-1}^{\,p})=(-1)^{n+1}Q_{n,s-1}^{\,p}
\qquad (1\le s\le n-1).
\]
In particular,
\[
c_1=(-1)^{n+1}Q_{n,0}^{\,p}.
\]
Since $Q_{n,0}$ is a unit in
\[
S=D_n[Q_{n,0}^{-1}],
\]
the element $c_1$ is a unit in $S$. Hence the Koszul complex $K_{n+1}(S)$ is contractible, or equivalently
\[
H_*\big(K_{n+1}(S)\big)=0.
\]
This example shows why one should not overstate the localized Koszul construction: after inverting $Q_{n,0}$, the most natural sequence can contain a unit.
\end{example}

The preceding example motivates retaining the coefficient sequence before inverting \(Q_{n,0}\).

\begin{proposition}[A non-localized Koszul model for \(i=n+1\)]\label{prop:nonlocalized-koszul-nplus1}
Assume \(n\ge2\). Let
\[
a_s:=(-1)^{n+1}Q_{n,s-1}^{\,p}\in D_n
\qquad (1\le s\le n-1).
\]
Then $(a_1,\dots,a_{n-1})$ is a $D_n$-regular sequence. Consequently, the Koszul complex
\[
K(a_1,\dots,a_{n-1};D_n)
\]
has homology
\[
H_0\cong
D_n/(Q_{n,0}^{\,p},Q_{n,1}^{\,p},\dots,Q_{n,n-2}^{\,p}),
\qquad
H_{>0}=0.
\]
\end{proposition}

\begin{proof}
The signs $(-1)^{n+1}$ are units in $\F_p$, so they do not affect regularity or the quotient. Since
\[
D_n=\F_p[Q_{n,0},Q_{n,1},\dots,Q_{n,n-1}]
\]
is a polynomial algebra, the variables
\[
Q_{n,0},Q_{n,1},\dots,Q_{n,n-2}
\]
form a regular sequence. Their $p$-th powers also form a regular sequence. Therefore
\[
(Q_{n,0}^{\,p},Q_{n,1}^{\,p},\dots,Q_{n,n-2}^{\,p})
\]
is $D_n$-regular, and hence so is $(a_1,\dots,a_{n-1})$. The stated homology follows from the standard exactness theorem for Koszul complexes on regular sequences.
\end{proof}

\begin{remark}
For \(n=1\), the sequence in Proposition~\ref{prop:nonlocalized-koszul-nplus1} is empty. Its Koszul complex is \(D_1\) in exterior degree \(0\), with
\[
H_0\cong D_1,
\qquad
H_{>0}=0.
\]
Thus the proposition's nonempty regular-sequence statement begins at \(n=2\).
\end{remark}

For a general $i$, one still has the following formal consequence.

\begin{proposition}[Formal Koszul consequence]\label{prop:formal-koszul}
Let $U\subset S$ be any multiplicative subset. If the sequence
\[
(c_1,\dots,c_{n-1})
\]
is $U^{-1}S$-regular and contains no unit in $U^{-1}S$, then the localized complex
\[
U^{-1}K_i(S)
=
\big(E(e_1,\dots,e_{n-1})\otimes U^{-1}S,d_i\big)
\]
has homology
\[
H_0\big(U^{-1}K_i(S)\big)
\cong
U^{-1}S/(c_1,\dots,c_{n-1}),
\qquad
H_{>0}\big(U^{-1}K_i(S)\big)=0.
\]
If one of the $c_s$ is a unit in $U^{-1}S$, then the same Koszul complex is contractible and all its homology vanishes.
\end{proposition}

\begin{proof}
After localization at $U$, the complex $K_i(S)$ becomes the standard Koszul complex on the image of the sequence $(c_1,\dots,c_{n-1})$ in $U^{-1}S$. If the localized sequence is regular and contains no unit, the standard exactness theorem for Koszul complexes on regular sequences gives
\[
H_{>0}\big(U^{-1}K_i(S)\big)=0,
\qquad
H_0\big(U^{-1}K_i(S)\big)\cong U^{-1}S/(c_1,\dots,c_{n-1}).
\]
If some $c_s$ is a unit, the ideal generated by the sequence is the whole ring and the Koszul complex is contractible. This proves the proposition.
\end{proof}

\begin{remark}[What is and is not proved here]\label{rem:what-not-proved}
Proposition~\ref{prop:formal-koszul} is a formal commutative-algebra statement. It does not prove that the sequence $(c_1,\dots,c_{n-1})$ is regular in any specific localization. Moreover, Example~\ref{ex:koszul-i-n+1} shows that the localization $D_n[Q_{n,0}^{-1}]$ can trivialize the Koszul complex by making one coefficient a unit. The non-localized Proposition~\ref{prop:nonlocalized-koszul-nplus1} gives a concrete regular sequence in the low extra case $i=n+1$, but it remains an ordinary Koszul calculation rather than a Margolis-homology computation. The point of the comparison with Tuan's work \cite{Tuan} is therefore structural: both settings naturally lead to Koszul complexes, but the differentials and the homological meanings are different.
\end{remark}

\section{Conclusion}

This paper has established a unified algebraic framework for the primitive Steenrod--Milnor action on the Dickson algebra by introducing the normalized derivation $\delta_i=(-1)^nQ_{n,0}^{-1}\StDeltai$. Starting from Sum's explicit generator formula, we demonstrated that $\delta_i$ preserves the polynomial subalgebra $D_n$ and cleanly transforms the action into the affine system $\delta_i(Q_{n,s})=A_s+BQ_{n,s}$. By expanding the non-commuting product of $Q_{n,0}$ and $\delta_i$ via unsigned Stirling numbers of the first kind, we completely determined the higher iterates of $\StDeltai$, culminating in the ambient nilpotence identity $(\StDeltai)^p=0$ on $D_n$.

A central structural discovery of this work is the restricted $p$-power relation $\delta_i^p=B^{p-1}\delta_i$. When $B\ne0$, the rescaled derivation $\partial_i=B^{-1}\delta_i$ is idempotent ($\partial_i^p=\partial_i$), which naturally induces an $\F_p$-weight decomposition of the localized Dickson algebra. By constructing explicit Lagrange projectors, we showed that this decomposition is multiplicative and elegantly isolates the invariant ratios in weight zero and $Q_{n,0}$ in weight one. This machinery provides exact kernel and image descriptions for both $\delta_i$ and $\StDeltai$, alongside closed iterate formulas on every weight space.

As concrete applications, this general formalism unifies and upgrades the classical first-order formulas for $1\le i\le n$ and Sum's low extra formulas for $i=n+1,n+2$. We highlighted the qualitative difference in the classical range: for $1\le i<n$, the normalized derivation acts as a scalar multiple of differentiation with respect to $Q_{n,i}$, while for $i=n$, it acts as the Euler derivation for the auxiliary polynomial degree. In both cases, the exact kernel and image are explicitly determined.

Finally, the normalized ratios give rise to an ordinary commutative-algebra Koszul complex defined by the coefficients $c_s=(-1)^nP_{n,i,s}^{\,p}$. We provided a careful structural comparison with recent Margolis-homology computations, emphasizing that $\StDeltai$ (lying on the polynomial $\xi$-side) is not generally square-zero, unlike the Milnor primitives $Q_j$ on the exterior $\tau$-side. While the non-localized case $i=n+1$ yields a concrete regular sequence, investigating the regularity of the coefficient sequence in the general case remains a natural and promising problem for future research.

\section*{Acknowledgments}
The author thanks Professor Geoffrey Powell for insightful comments and valuable mathematical discussions. The author also thanks the anonymous referee for a careful reading and constructive suggestions.

\end{document}